\global\def\draftcontrol{0}

   \def\versionno{ Integrality }

\catcode`\@=11

\expandafter\ifx\csname draftcontrol\endcsname\relax\global\def\draftcontrol{0} 
\fi 

{\count255=\time\divide\count255 by 60 
\xdef\hourmin{\number\count255} 
\multiply\count255 by-60\advance\count255 by\time 
\xdef\hourmin{\hourmin:\ifnum\count255<10 0\fi\the\count255}} 
\def\draftdate{\number\month/\number\day/\number\year\ \ \ \hourmin } 

\newcommand\makepapertitle{\par

  \begingroup 
    \renewcommand\thefootnote{\@fnsymbol\c@footnote}%
    \def\@makefnmark{\rlap{\@textsuperscript{\normalfont\@thefnmark}}}%
    \long\def\@makefntext##1{\parindent 1em\noindent 
            \hb@xt@1.8em{%
                \hss\@textsuperscript{\normalfont\@thefnmark}}##1}%
     \newpage 
     \global\@topnum\z@   
     \@makepapertitle 
     \thispagestyle{empty}\@thanks 
  \endgroup 
  \setcounter{footnote}{0}%
  \global\let\thanks\relax 
  \global\let\makepapertitle\relax 
  \global\let\@makepapertitle\relax 
  \global\let\@thanks\@empty 
  \global\let\@author\@empty 
  \global\let\@date\@empty 
  \global\let\@title\@empty 
  \global\let\title\relax 
  \global\let\author\relax 
  \global\let\date\relax 
  \global\let\and\relax 
  \def\version{\let\version\@version\@gobble} 
} 
\def\@makepapertitle{%
  \newpage 
   \ifnum\draftcontrol=1 {} 
   \version\versionno 
   \vskip 5.5em%
   \else 
   \hfill\hbox to 3cm {\parbox{5.5cm}{\@pubnum}\hss}%
   \vskip 6.5em%
   \fi 
   \begin{center}%
   \let \footnote \thanks 
      {\hskip -0\textwidth \hbox to 1\textwidth%
        {\centerline{\Large\bf{\noindent%
	\parbox[t]{1.3\textwidth}{\begin{center}\@title\end{center}}}}}}%
     \vskip 1.5em%
     {\normalsize
       \lineskip .5em%
       \begin{tabular}[t]{c}%
         \@author 
       \end{tabular}\par}%
     \vskip 1.5em%
     {\@bstract}%
     \end{center}%
     \vfill
     \@date%
     \vskip 1.5em%
   \par 
} 

\gdef\@pubnum{} 
\def\pubnum#1{%
  \gdef\@pubnum{#1}} 

\gdef\@bstract{} 
\def\Abstract#1{%
  \gdef\@bstract{%
   \parbox{\textwidth-0pc}{%
   \centerline{\bf Abstract}\penalty1000 
   \noindent
   \renewcommand\baselinestretch{1.0} 
   {#1}}} 
} 

\gdef\@email{}
\def\email#1{%
   \gdef\@email{%
   Email: {\tt #1}}
}

\def\ps@paper{\let\@mkboth\@gobbletwo%
     \ifnum\draftcontrol=1 
        \def\@oddfoot{\hbox to \textwidth{\tiny \versionno \hfil\tiny\draftdate}%
        \hskip -\textwidth \hbox to \textwidth{\hfil\rm\thepage\hfil}}%
     \else\def\@oddfoot{\hbox to \textwidth{\hfil\rm\thepage\hfil}} 
     \fi 
     \let\@evenfoot\@oddfoot 
} 

\def\body{\clearpage 
          \pagestyle{paper} 
        } 
\newenvironment{acknowledgments}{%
\vskip 3.25ex 
\addcontentsline{toc}{section}{Acknowledgments}
\noindent {\bf Acknowledgments} 
} 

\def\@version#1{\ifnum\draftcontrol=1 
\typeout{}\typeout{#1}\typeout{} 
\vskip3mm\centerline{\hbox{\fbox{\normalsize{\tt DRAFT -- #1 -- } 
                   {\draftdate}}}}\vskip3mm 
\fi} 
\let\version\@version 
\long\def\eqlabel#1{\ifnum\draftcontrol=1 
                    \tag@false  
                    \tag*{(\theequation) \hbox to -0.2cm{\hspace{0cm}\small{#1}\hss}} 
                    \refstepcounter{equation}  
                    \edef\@currentlabel{\theequation} 
                    \ltx@label{#1}          
                    \else 
                    \label{#1} 
                    \fi 
                    } 
\let\st@bibitem\@bibitem 
\let\st@lbibitem\@lbibitem 
\ifnum\draftcontrol=1 
  \def\@bibitem#1{%
    \st@bibitem{#1}\a@@label{#1}\ignorespaces}
  \def\@lbibitem[#1]#2{%
    \st@lbibitem[#1]{#2}\a@@label{#2}\ignorespaces} 
  \def\a@@label#1{%
    \gdef\a@lab{\smash{\normalfont\small#1}} 
    \ifvmode 
      \if@inlabel 
        \global\setbox\@labels\hbox{%
          \llap{\a@lab\let\a@lab\relax 
                \kern\@totalleftmargin\kern\marginparsep}%
          \box\@labels}%
      \fi 
    \fi} 
\fi 

\documentclass[12pt,letterpaper]{article} 

\usepackage{amsmath,bm,amsfonts,amssymb,array,calc,amsthm,rotating}
\usepackage{epsfig,psfrag} 
\usepackage{amsthm}
\usepackage{wasysym}
\usepackage{graphicx}
\usepackage{color}
\usepackage[colorlinks=true]{hyperref}
\usepackage[stable]{footmisc}

\renewcommand\baselinestretch{1.25} 
\renewcommand\section{\@startsection {section}{1}{\z@}%
                                   {-3.5ex \@plus -1ex \@minus -.2ex}%
                                   {2.3ex \@plus.2ex}%
                                   {\normalfont\large\bfseries}} 
\renewcommand\subsection{\@startsection{subsection}{2}{\z@}%
                                   {-3.25ex\@plus -1ex \@minus -.2ex}%
                                   {1.5ex \@plus .2ex}%
                                   {\normalfont\normalsize\bfseries}} 
\renewcommand\subsubsection{\@startsection{subsubsection}{3}{\z@}%
                                   {-3.25ex\@plus -1ex \@minus -.2ex}%
                                   {1.5ex \@plus .2ex}%
                                   {\normalfont\normalsize\it}} 
\renewcommand\paragraph{\@startsection{paragraph}{4}{\z@}%
                                   {-1.25ex\@plus -1ex \@minus -.2ex}%
                                   {-1.5ex \@plus .2ex}%
                                   {\normalfont\normalsize\bf}} 
\renewcommand\subparagraph{\@startsection{subparagraph}{5}{\z@}%
                                   {.5ex\@plus -1ex \@minus -.2ex}%
                                   {.1ex \@plus .2ex}%
                                   {\normalfont\normalsize\it}}

\newtheorem{theorem}{Theorem}
\newtheorem{proposition}[theorem]{Proposition}

\newtheorem{lemma}[theorem]{Lemma}

\theoremstyle{definition}

\theoremstyle{remark}
\newtheorem{remark}[theorem]{Remark}

\numberwithin{equation}{section}

\long\def\@makecaption#1#2{%
  \vskip\abovecaptionskip
  \sbox\@tempboxa{{\bf #1:} #2}%
  \ifdim \wd\@tempboxa >\hsize
    {\small\bf #1:} {\small #2}\par
  \else
    \global \@minipagefalse
    \hb@xt@\hsize{\hfil\box\@tempboxa\hfil}%
  \fi
  \vskip\belowcaptionskip}

\renewcommand*\l@section[2]{%
  \ifnum \c@tocdepth >\z@
    \addpenalty\@secpenalty
    \addvspace{.5em \@plus\p@}%
    \setlength\@tempdima{1.5em}%
    \begingroup
      \parindent \z@ \rightskip \@pnumwidth
      \parfillskip -\@pnumwidth
      \leavevmode \bfseries
      \advance\leftskip\@tempdima
      \hskip -\leftskip
      #1\nobreak\hfil \nobreak\hb@xt@\@pnumwidth{\hss #2}\par
    \endgroup
  \fi}
\renewcommand*\l@subsection{\addvspace{.0em \@plus\p@}\@dottedtocline{2}{1.5em}{2.3em}}
\renewcommand*\l@subsubsection{\addvspace{-.2em \@plus\p@}\@dottedtocline{3}{3.8em}{3.2em}}

\def\hepth#1{\href{http://xxx.arxiv.org/abs/hep-th/#1}{{arXiv:hep-th/#1}}}

\def\alggeom#1{\href{http://xxx.arxiv.org/abs/alg-geom/#1}{{arXiv:alg-geom/#1}}}
\def\arxiv#1#2{\href{http://xxx.arxiv.org/abs/#1}{{arXiv:#1 [#2]}}}

\definecolor{refcol}{rgb}{0.0,0.0,0.2}
\definecolor{eqcol}{rgb}{.2,0,0}
\definecolor{purple}{cmyk}{0,1,0,0}

\gdef\@citecolor{refcol}
\gdef\@linkcolor{eqcol}
\gdef\@urlcolor{refcol}
\def\colorlinkspurple{\gdef\@urlcolor{purple}}
\def\colorlinksblue{\gdef\@urlcolor{blue}}
\def\colorlinksred{\gdef\@urlcolor{red}}

\def\ie{{\it i.e.}} 
 
\def\etc{{\it etc.}}
\def\cf{{\it cf.}}

\def\revise#1       {\raisebox{-0em}{\rule{3pt}{1em}}%
                     \marginpar{\raisebox{.5em}{\vrule width3pt\ 
                     \vrule width0pt height 0pt depth0.5em 
                     \hbox to 0cm{\hspace{0cm}{%
                     \parbox[t]{4em}{\raggedright\footnotesize{#1}}}\hss}}}}

\def\cala         {{\cal A}} 
\def\calb         {{\cal B}} 
\def\calc         {{\cal C}} 
\def\cald         {{\cal D}} 
 
\def\calf         {{\cal F}} 
 
\def\calh         {{\cal H}} 
\def\cali         {{\cal I}} 
 
\def\calk         {{\cal K}} 
\def\call         {{\cal L}} 
 
\def\caln         {{\cal N}} 
\def\calo         {{\cal O}} 
\def\calp         {{\cal P}}

\def\cals         {{\cal S}} 
\def\calt         {{\cal T}} 
 
\def\calv         {{\cal V}} 
\def\calw         {{\cal W}} 

\def\caly         {{\cal Y}}

\def\C      {{\mathbb C}}

\def\Q    {{\mathbb Q}} 
\def\Qbar {{\overline \Q}}
 
\def\Z          {{\mathbb Z}}

\def\G {{\calo}}

\def\del          {\partial} 
 
\def\ee           {{\it e}} 
\def\ii           {{\it i}}

\def\id           {{\rm id}}

\newcommand\topa[2]{\genfrac{}{}{0pt}{2}{\scriptstyle #1}{\scriptstyle #2}}

\def\sqr#1#2{{\vcenter{\vbox{\hrule height.#2pt   
 \hbox{\vrule width.#2pt height#1pt \kern#1pt 
 \vrule width.#2pt}\hrule height.#2pt}}}}

\let\dint\varint
\def\fp{\mathfrak{p}}
\def\frob{\operatorname{Frob}}
\def\spec{\operatorname{Spec}}
\def\spann{\operatorname{Span}}

\def\ker{\operatorname{Ker}}

\catcode`\@=12 

\begin{document} 


\title{Integrality of Framing and Geometric Origin of 2-functions \\
(with algebraic coefficients)}

\pubnum{%
\arxiv{1702.07135}{hep-th}
}
\date{February 2017}

\author{
Albert Schwarz$^{\dag}$, Vadim Vologodsky$^{\ddag}$, and Johannes Walcher$^{\#}$ \\[0.2cm]
\it $^{\dag}$Department of Mathematics,\\ 
\it University of California at Davis, California, USA \\[.1cm]
\it $^{\ddag}$
National Research University, Higher School of Economics, Russian Federation 
\\ \rm and \\
\it Department of Mathematics, University of Oregon, Eugene, Oregon, USA \\[.1cm]
\it $^{\#}$Mathematical Institute,\\
\it Ruprecht-Karls-Universit\"at Heidelberg, Germany}

\Abstract{
We say that a  formal power series $\sum a_nz^n$ with rational coefficients is a 
$2$-function if the numerator of the fraction $a_{n/p}-p^2 a_n$ is divisible by $p^2$ 
for every prime number $p$. One can prove that  $2$-functions with rational 
coefficients appear as building block of BPS generating functions in topological 
string theory. Using the Frobenius map  we define 2-functions with coefficients
in algebraic number fields.
We establish two results pertaining to  these functions. First, we show that the 
class of 2-functions is
closed under the so-called framing operation (related to compositional inverse of
power series). Second, we show that 2-functions arise naturally in geometry
as $q$-expansion of the truncated normal function associated with an algebraic
cycle extending a degenerating family of Calabi-Yau 3-folds.}

\makepapertitle

\body

\version\versionno

\vskip 1em

\tableofcontents

\section{Introductions}
\label{introduction}

\footnote{The main results of this paper are stated in subsections \ref{definitions}
(Theorem \ref{framtheo}) and \ref{intdef} (Theorem \ref{origin}).}
The classical ``mirror principle'' as developed in the early 1990's, states that
the Gromov-Witten theory of a Calabi-Yau threefold $X$ can be encoded in Hodge
theoretic data of a mirror manifold $\caly\to B$, which is a family of Calabi-Yau
threefolds, expanded around a maximal degeneration point $0\in \overline{B}\setminus B$.
The physicist's intuition behind this statement is the equivalence of the effective
physical theories obtained by compactifying string theory on the two different
manifolds.

Beginning in the late 1990's, developments based on other physical dualities 
(involving M-theory) have shown that Gromov-Witten theory can be rewritten
in terms of mathematical invariants enumerating stable objects in D-brane 
categories that can be attached to either manifold. These invariants capture 
the physical notion of ``degeneracy of BPS states'' in the effective theory.

An important feature of this reformulation is that while Gromov-Witten
invariants are a priori rational numbers, they can in fact be expressed
as linear combination (with fixed denominator) of integers, which moreover 
have the interpretation as (graded) dimensions of vector spaces (the physical 
Hilbert space of BPS states). One might say, the invariants are automatically 
``categorified''.

On the Gromov-Witten side (the A-model of mirror symmetry), many of these 
reformulations have been elevated to mathematical theorems in the recent years. 
From the point of view of the mirror manifold (the B-model), 
the integrality underlying Gromov-Witten theory is a rather non-trivial property 
of the Hodge theoretic expansion around the maximal degeneration point. 

Before most of the A-model proofs were available, it had been shown in \cite{ksv,sv1,vadim} 
that integrality can be established independently in the B-model by passing through
the world of $p$-adic Hodge theory. The basic idea is to show that, for any given
prime number $p$, the reformulated invariants (which are a priori rational numbers)
have denominators not divisible by $p$. In other words, one establishes certain 
$\bmod p$ congruences between the expansion coefficients of the periods. If these
congruences hold for all primes $p$, then the reformulated invariants themselves have 
to be integral.

The relevance of $p$-adic methods is quite intriguing as it connects the physics 
of Calabi-Yau manifolds to a number of interesting topics in arithmetic geometry 
and number theory. On the one hand, the method is naively rather unnatural from 
the physical point of view. (The idea that our finite experience of the physical 
world can be accounted for in integers is old and well-known, but prime numbers do not 
normally play a role in it.) On the other hand, it is not immediately clear how the 
number theoretic methods mesh with ``categorification'', what the underlying
integers are counting in the B-model, and whether they are naturally dimensions of 
some vector spaces. Filling these gaps in the current understanding
clearly is an opportunity to bridge between the two subjects, supersymmetric
quantum theory, and number theory.

\medskip

In another recent development \cite{arithmetic}, it was pointed out that a certain
class of extensions of the Hodge theoretic situation, that is very natural
from both the physical and mathematical point of view, generically leads to
expansion coefficients that are no longer rational, but instead take values
in an algebraic number field, fixed for each such situation. This raises
the intriguing question whether it is possible to interpret such irrational
invariants as ``enumerative'' in a generalized sense or whether some other
assumption has broken down. To us, the categorical equivalence (which, at least for 
the quintic, has now been proven \cite{sherridan}) and the extensive experience in 
many other situations (most closely related to ours are \cite{agva,akv,opening}) 
suggest that the mirror principle is of very general validity. Therefore, we believe
that a suitably applied Gromov-Witten theory should explain or otherwise accomodate 
the irrationality of the invariants. It is clear that the relevant A-model 
situation involves 
the enumerative geometry of generic objects of the Fukaya category, but the
details are unknown.\footnote{Some speculations were offered in 
\cite{arithmetic}, and in various talks
given by the third-named author. See also section \ref{physmotiv}. An interesting 
possibility, advocated by C.~Vafa, is that we are not working around the mirror of a 
fully classical regime, and that rationality in the B-model will be restored by a 
further degeneration. We expect that a combination of HMS and SYZ will eventually
shed light on this mystery.}

Perhaps the strongest evidence that such an explanation should exist is the fact that
the expansion displays an integrality that is a generalization of that underlying 
the rational B-model (and proven by the $p$-adic methods in \cite{ksv,sv1,vadim,sv2})
to the situation with algebraic expansion coefficients. An important feature
of the general setup is that (when the Galois group is non-abelian), one needs 
to invoke $p$-adic considerations to even {\it formulate} the statement of
integrality (and of course, also in the proof, see below). We then see two
possibilities for relaxing the tension with enumerative geometry. Either the 
physics (or A-model) explanation does depend on the notion of a prime number 
as well, or it knows implicitly how to eliminate (or ``integrate out'') $p$
in a way that is so far unknown to mathematicians. Either resolution would be 
very interesting.

\medskip

This paper is a result of combining the $p$-adic proofs of integrality of instanton 
numbers \cite{ksv,sv1,vadim} and of integrality of the number of holomorphic disks 
\cite{sv2} with the recent observations \cite{arithmetic} about the irrationality
of the Hodge theoretic expansion in the generic extended situation. Namely, we will
prove the integrality statement of \cite{arithmetic}. We hope that eventually these
results and the method of proof
will help to clarify the A-model interpretation of the irrationality 
(as well as the integrality), and perhaps point to a deeper physical and mathematical
message. At a preliminary stage, we were led to introduce and study, independently 
of the geometric context, a certain class of power series that we dub ``2-functions'' 
(where, more generally, $2$ could also be replaced by some other positive integer $s$). 
In particular, we show that the class of 2-functions (but not general $s$-functions)
is closed under the framing operation known from
local open string mirror symmetry \cite{akv} (where it is mirror to the framing
of knots in 3-manifolds, hence the name). This part is a generalization of the
previous paper \cite{svw1} to the situation with arbitrary algebraic coefficients.
(In fact, the proof immediately generalizes to a completely abstract situation, for
which however we have no use at the moment.)

Thus, the paper is naturally divided in two parts which are logically independent
from each other. The main results are stated in section \ref{framing} (integrality of 
framing with algebraic coefficients) and in section \ref{geosit} (geometric origin of 
2-functions). In the rest of this somewhat lengthy introductory section, we 
offer some mathematical 
and physical motivation which we expect to provide a part of the bigger picture.%
\footnote{This material, useful for exposition to the mixed readership, is very 
elementary, but only partly self-contained. It is hardly necessary for an 
understanding of the technical content of the paper.}

\subsection{Motivation for Physicists}
\label{physmotiv}

In this subsection, we give a quick review of a few basic notions from algebraic 
number theory, and explain some reasons we think they might play a role 
in physics.

To begin with, we recall that an {\it algebraic number}, $x$, is simply a root of a
(non-constant) polynomial with integer coefficients. In other words $P(x)=0$ where 
$P= a_n x^n+a_{n-1} x^{n-1}+\cdots +a_0\in\Z[x]$ with $a_k\in\Z$, $a_n\neq 0$. The 
field of all algebraic numbers is denoted $\Qbar$. Given $x\in\Qbar$, the  
polynomial $P$ of smallest degree of which $x$ is a root (which is unique if we 
require the coefficients to be co-prime) is known as the minimal polynomial of $x$. 
By adjoining $x$ to $\Q$, we obtain an {\it algebraic number field} (a finite extension 
of the field $\Q$ of rational numbers), $K=\Q(x)= \Q[x]/P$.

Physicists might be used to thinking of algebraic numbers simply as complex 
numbers. The more abstract definition however does not specify which of the 
$n$ roots of $P$ to call $x$, and nothing in the algebra depends on this choice 
(if $P$ is irreducible). In physics language, one might say that picking one of 
the roots (to ``embed'' $K$ into $\C$) amounts to {\it breaking the symmetries} 
of the problem.

More formally, given an algebraic number field $K$ generated by an algebraic number 
$x$, which we think of as any one of the roots of polynomial $P$, it is of interest to 
ask whether $K$ contains any other roots of $P$. If $K$ contains all roots of $P$, 
then $K$ is said to be Galois over $\Q$. This is equivalent to the statement 
that if we denote by ${\rm Gal}(K/\Q)$ the (Galois) group of automorphisms 
of $K$ that leave $\Q$ invariant then $\Q$ is the fixed field of ${\rm Gal}(K/\Q)$.
If $K=\Q(x)$ is not Galois, we may Galois-close the field by adjoining all the other 
roots of the minimal polynomial. The resulting field, known as the splitting field
of $P$, is generically of higher degree.

Thinking of all roots of $P$ on equal footing respects the Galois symmetries. A
fundamental observation is that the more generic the polynomial, the larger the 
Galois group of its splitting field.

For example, if $x^2+3=0$, $\Q(x)=\Q(\sqrt{-3})\cong \Q(\ee^{2\pi\ii/3})$ 
with Galois 
group $\Z/2$ generated by $\sqrt{-3}\mapsto -\sqrt{-3}$. For a different example, 
$\Q(5^{1/3})$ is {\it not} a Galois extension. This is because the other two roots 
of the minimal polynomial $x^3-5$, which are of course $\ee^{2\pi\ii/3} 5^{1/3}$
and $\ee^{4\pi\ii/3} 5^{1/3}$ cannot be expressed algebraically in terms of $5^{1/3}$. 
This is resolved by adjoining $\sqrt{-3}$, and so we learn that the Galois closure
is $\Q(5^{1/3},\sqrt{-3})$, with Galois group $S_3$. This means simply that
the three roots are algebraically on equal footing, and is the generic situation
with a cubic polynomial. An example of a cubic extension that is Galois is provided 
by $x^3+x^2-2x-1$. In that case, the other two roots can be written in terms of $x$ 
alone, as $x^2-2$, and $1-x-x^2$. These algebraic relations between the roots break
the Galois group from $S_3$ down to $\Z/3$.

We record two more elementary definitions. First, among all algebraic numbers, those 
whose minimal polynomial $P$ has leading coefficient $a_n=1$ are known as {\it algebraic 
integers}. They play a similar role in $K$ as the ordinary integers $\Z$ play in $\Q$. 
In particular, the algebraic integers form a ring, which we denote by $\G_K$ or simply
$\G$ if $K$ is clear from the context.
Second, the {\it discriminant} of the extension $D(K/\Q)\in \Z$, is a (rational) 
integer which gives a measure of the size of $\G$ relative to $\Z$. We won't define it 
precisely here but note that it divides the discriminant of the minimal polynomial $P$ 
of an integral generator $x$ (and sometimes $D(K/\Q)$ is equal to the discriminant of
$P$).

\medskip

Now let us ask: How might any of this be relevant to (mathematical) physics? It
is a familiar fact that supersymmetry constrains configuration spaces of 
supersymmetric field and string theories to be complex (Kahler, super-) manifolds. This 
is true for both the space of continuous off-shell fields, as well as the on-shell 
spaces of supersymmetric vacua. It is equally familiar that many physical questions 
about these theories can be answered by viewing the relevant spaces more abstractly
as algebraic varieties, and using methods from algebraic geometry. 

The results of \cite{arithmetic} and of the present paper, however, reveal
that to describe the kinematics (and some dynamics) of certain situations
involving (close to) minimal supersymmetry, it is essential to understand
the field of definition of the underlying spaces, and to separate the
algebraic properties from the complex analytic ones. Our physical interpretation
is that in these situations, the breaking of supersymmetry should generally 
be thought of as an ``extension of algebraic structure'', and that the minimal 
amount of structure in the vacuum is the field of definition (or more precisely, the 
``semi-classical residue field''). The Galois group then quite literally acts on the
vacua, as well as (more conjecturally) on the space of physical states. 

\medskip

To explain this in more detail, we recall that in supersymmetric 
field theories with $4$ supercharges (corresponding to $\caln=1$
in 4 dimensions), the dynamics of chiral fields $\Phi$ (whose lowest component is a 
complex scalar field) are governed by two types of terms in the supersymmetric 
Lagrangian: The Kahler potential $\calk(\Phi,\bar \Phi)$ that determines the kinetic 
terms in the bosonic Lagrangian, and the superpotential $\calw(\Phi)$ that determines 
the potential terms. While the Kahler potential is quite flexible, the superpotential 
has to be holomorphic (as well as being constrained by any global and local symmetries 
that might be present). Therefore, if our goal is to connect the physics of $\caln=1$ 
supersymmetric field theories with algebra and algebraic geometry (say we want to 
elucidate the physical content of an $\caln=1$ supersymmetric compactification
of string theory on an algebraic variety), it is natural to focus on the superpotential 
as one of the exactly calculable quantities.

But how much invariant physical information is really contained in the superpotential alone,
even assuming we could calculate it exactly? Clearly, we should be looking at supersymmetric
vacua, in other words, expand around a critical point of the superpotential. However, 
even in supersymmetric vacua, statements about physical masses and about Yukawa and 
higher order interactions depend on the proper normalization of the kinetic terms, 
hence the Kahler potential. The simplest quantity that does not depend on $\calk$ is
the constant term in the expansion, in other words, the critical value of the 
superpotential, 
\begin{equation}
\eqlabel{critical}
{\rm Crit}(\calw) = \{ \calw|_{\rm \del\calw = 0} \}
\end{equation}
More precisely, since (in the absence of gravity) $\calw$ is defined only up to an
additive constant, the truly invariant quantities are the {\it differences of the 
critical values}. These differences are known, by one of the most elementary BPS bounds, 
to give the tension (or masses, in 2 space-time dimensions) of supersymmetric domain walls 
(BPS solitons) interpolating between the various supersymmetric vacua. If 
$\Phi^{(i)}$ and $\Phi^{(j)}$ are two critical points of $\calw$, with critical
values $\calw^{(i)}$ and $\calw^{(j)}$, co-dimension one BPS defects interpolating 
between $\Phi^{(i)}$ and $\Phi^{(j)}$ have tension 
\begin{equation}
\eqlabel{tension}
m_{ij} = |\calt_{ij}|:=|\calw^{(j)}- \calw^{(i)}| \,,
\end{equation} 
while $\alpha_{ij}:=\arg (\calt_{ij})$ measures the linear combination
of supersymmetries preserved by the defect (assuming that $\calt_{ij}\neq 0$).
As a secondary quantity, it is of interest to consider the {\it degeneracy} of
such BPS defects, which is the dimension of the corresponding Hilbert space
$\calh_{ij}^{\rm BPS}$.

\medskip

To connect this with algebra and field extensions, let us {\it assume} that for some 
a priori reasons, {\it the superpotential is constrained to be polynomial with integer 
coefficients}. This will likely sound like a strong assumption, and we have no
control over the class of situations in which it holds. What matters
for us in the end is that the assumption seems to be satisfied in the examples 
coming from D-branes on Calabi-Yau manifolds (see section \ref{geosit} or ref.\ 
\cite{arithmetic,svw1}). 
Temporarily, one can think of a superpotential that is generated by instantons (counted
by integer coefficients), of which only a finite number are relevant for finding the
critical points (so that it is polynomial). In a more general version, we like to
think that the underlying integral structure comes from a bulk theory with extended 
supersymmetry into which our $4$-supercharge theory is embedded.

In any event, if $\calw\in \Z[\Phi]$, it is easy to see that the critical values 
\eqref{critical} will be algebraic numbers, \ie, they will be roots of some
(other!) polynomial $P$ with integral coefficients. We emphasize that although
$P$ is of course determined by $\calw$, the two polynomials are conceptually and
algebraically distinct. For instance, it is not immediately clear whether any
$P$ can appear as we vary $\calw$, \ie, whether any algebraic number can be
obtained as the critical value of a polynomial with rational coefficients.

In thinking about this situation, one is naturally led to wonder whether the
Galois symmetries of (the splitting field of) $P$ have any physical import.%
\footnote{We are aware that similar ideas have been formulated in 
\cite{ferrari,cachazo}.}
At first sight, the appearance of the absolute value (the Archimedean norm)
in \eqref{tension} looks like evidence that the physically relevant geometry is
just that of the complex plane. But again, if we accept that we only want to look 
at the {\it algebraic properties}, we ought to not separate $\calt_{ij}$ 
into $m_{ij}$ and $\alpha_{ij}$, and the vacua of the theory are indeed 
related by the Galois symmetry of the polynomial $P$. We propose that this
symmetry carries interesting physical information about the theory.
More specifically, we expect that ${\rm Gal}(K/\Q)$ will act on the space 
of BPS states, $\oplus_{i,j} \calh_{ij}^{\rm BPS}$. 

The present paper constitutes some indirect evidence for this proposal. One of the
reasons that we are not able to state a more precise conjecture is that the 
formulation of our results involves one more ingredient for which we presently 
have no physical interpretation at all: This ingredient is the notion of a prime number $p$. 
Whether such primes admit a physical interpretation, or whether it is possible to 
eliminate the primes from the mathematical formulation, remains to be seen. 
Either outcome would be very interesting.

\medskip

To conclude this subsection, we recall why prime numbers are useful for
elucidating the structure of algebraic number fields. The main idea (which 
has no physical counterpart) is to treat a prime $p$ as a ``small parameter'', 
and to study ${\rm Gal}(K/\Q)$ $p$-adically, \ie, in an expansion in this 
small parameter.

We recall Fermat's Little Theorem, which
states that if $p$ is prime, and $a\in\Z$ any integer, then
\begin{equation}
\eqlabel{flt}
a^p \equiv a \bmod p
\end{equation}
As a consequence, if $P=a_n x^n + \cdots + a_0\in\Z[x]$ and $P(x)=0$,
then,
\begin{equation}
\eqlabel{thenmodp}
P(x^p) = \sum a_k x^{k p} \equiv \sum (a_k x^k)^p \equiv \bigl( \sum a_k x^k \bigr)^p 
\equiv 0 \bmod p
\end{equation}
Thus, given a prime $p$, we can obtain a ``first approximation'' to another root of 
$P$ by simply raising $x$ to the $p$-th power. This {\it Frobenius operation}
is of finite order $\bmod p$ and can be used to identify certain interesting
subgroups of the Galois group. We defer precise definitions to section \ref{framing},
and here only point out the important dichotomy that arises between abelian and
non-abelian Galois group. In the former case, the $\bmod\ p$ Frobenius element of 
${\rm Gal}(K/\Q)$ is canonically determined and the Frobenius elements for different 
primes all commute with one another. In the latter, non-abelian case, the Frobenius 
element only defines a conjugacy class in ${\rm Gal}(K/\Q)$, which moreover varies 
in a poorely controlled way with $p$.

To give some examples for practice, in the field $\Q(\sqrt{-3})$ (of Galois group 
$\Z/2$), one 
finds that for $p>3$, $\sqrt{-3}^p \equiv \sqrt{-3}\bmod p$ if $p\equiv 1\bmod 3$, 
while $\sqrt{-3}^p\equiv -\sqrt{-3} \bmod p$ if $p\equiv 2\bmod 3$. This regularity 
is a consequence 
of Gauss' quadratic reciprocity, and the Frobenius elements are the trivial or non-trivial 
element of the Galois group, respectively. On the other hand, let us consider 
$\Q(5^{1/3})$, which is not a Galois extension. One might well show that for 
$5<p\equiv 1\bmod 3$, $(5^{1/3})^{p-1}$ is always a cube root of unity $\bmod p$, but it 
is not possible to predict whether it will be a non-trivial cube root or not. For 
instance $(5^{1/3})^{7-1}=4\bmod 7$ and $4^3=1 \bmod 7$ (the Frobenius element
generates a $\Z/3$ subgroup of $S_3$), while $(5^{1/3})^{13-1}=1\bmod 13$ (and the
Frobenius is trivial). The Frobenius elements at primes $p$ with $p\equiv 2\bmod 3$
generate the odd permutations in $S_3$.

Beginning in the next subsection, we will inquire about ways to ``go to next order
in $p$'', \ie, to find roots of $P$ $\bmod p^2$. It will be seen that this is easy 
to do as long as we fix $p$, but that the non-commutativity of the Frobenius elements 
presents a obstacle for eliminating $p$ from the mathematical formalism.

\subsection{Motivation for Mathematicians}
\label{mathmotiv}

The main character of this paper are what we call 2-functions, certain (formal) 
power series with properties given in section \ref{framing}. The background in
physics and mirror symmetry is explained elsewhere. In this subsection, we explain 
in an informal way what these definitions can achieve for mathematics.

As above, we let $x\in \G_{\Qbar}$ be an algebraic integer, and $K=\Q(x)$ be the 
number field generated by $x$. We denote by $P\in\Z[x]$ the minimal polynomial 
of $x$. More generally, for $y\in \G_K$, the ring of integers in $K$, we'll let 
$P_y\in\Z[y]$ be the minimal polynomial of $y$.

Fixing an embedding $K\hookrightarrow \C$, we can think of $x$ as a complex number.
Let us consider, for $z$ in a neighborhood of $0\in\C$, the Mercator series
\begin{equation}
\eqlabel{log}
  -\log(1 - x z)
= \sum_{k=1}^\infty \frac{\rho_k(x)}{k} z^k
\end{equation}
where $\rho_k(y):=y^k$. The expansion of course converges in a neigborhood of $0$ 
(depending on the chosen embedding $K\hookrightarrow \C$), and the function can be 
analytically continued throughout some slit complex plane. The coefficients $\rho_k(x)$ 
possess the following properties:\footnote{In all statements below, we shall 
assume that $k$ is co-prime with the discriminant $D(K/\Q)$.}
\\[.2em]
(i) For $k=p$ prime, $P_x(\rho_p(x))= 0 \bmod p$. 
\\[.2em]
(ii) When $k=p^r$ is a prime power, we have 
$P_{\rho_{p^{r-1}}(x)}(\rho_{p^r}(x)) = 0 \bmod p^r$.
\\[.2em]
(iii) $\rho_{k_1k_2}(x) = \rho_{k_1}(\rho_{k_2}(x))=\rho_{k_2}(\rho_{k_1}(x))$.

\vskip .3em 

To be sure: (i) follows from Fermat's little theorem eq.\ \eqref{flt}, 
see eq.\ \eqref{thenmodp}. Similarly, (ii) (of which (i) is a special case) follows 
from Euler's generalization of Fermat's 
theorem: If $a=b\bmod p^{r-1}$, then $a^p=b^p\bmod p^{r}$. And while the
multiplicativity (iii) 
is of course trivial, we list it here because of the generalizations below. 
More precisely, the property we generalize is the following somewhat less 
trivial-looking corollary,
\\[.2em]
(iii') If $k= p^r k'$ with
$k'$ not divisible by $p$, then
$$
P_{\rho_{k/p}(x)}(\rho_k(x)) = 0 \bmod p^r
$$

We can qualitatively summarize these properties by saying that, as we vary $p$,
the power series \eqref{log} bundles together information about $\bmod p$ arithmetic 
in the number field $K=\Q(x)$. The results about 2-functions that we obtain in 
the present paper make the following question seem like a possible starting point
to motivate their study:

\begin{center}
\framebox{\parbox{\textwidth-1cm}{
Is it possible to ``integrate'' (\ref{log}) in such a way that the properties of 
$\rho_k$ are lifted modulo higher powers of $k$?}}
\end{center}

We illustrate what we mean in the first non-trivial instance, which is an improvement
of the above conditions from holding $\bmod p$ to $\bmod p^2$: Given $x\in \G_K$, we 
are looking for a collection of coefficients
\begin{equation}
\sigma_k(x) \in \G_K
\end{equation}
such that
\\
(i)$_2$ For $k=p$ (unramified) prime, $\sigma_p(x) = x^p\bmod p$, and
\begin{equation}
\eqlabel{i2}
P_x(\sigma_p(x)) = 0 \bmod p^2
\end{equation}
(ii)$_2$ For $k=p^r$ a prime power, $\sigma_{p^r}(x) = (\sigma_{p^{r-1}}(x))^p\bmod p$,
and
\begin{equation}
\eqlabel{ii2}
P_{\sigma_{p^{r-1}}(x)}(\sigma_{p^r}(x)) = 0 \bmod p^{2r}
\end{equation}
(iii)$_2$ For general $k$ (co-prime with discriminant of $K/\Q$), and any
$p|k$, we have $\sigma_k(x) = (\sigma_{k/p}(x))^p\bmod p$, and letting 
$e_p={\rm ord}_p(k)$ be the largest power of $p$ dividing $k$, 
\begin{equation}
\eqlabel{iii2}
P_{\sigma_{k/p}(x)} (\sigma_k(x)) = 0 \bmod p^{2 e_p}
\end{equation}
We remark that (iii)$_2$ is the natural lift of (iii) in the sense that the
$\rho_k$ satisfy its analogue $\bmod p^{e_p}$ (see (iii')), but the $\sigma_k$ 
(as maps $\G_K\to \G_K$) 
cannot be strictly multiplicative in general. In this formulation, of course 
(i)$_2$ and (ii)$_2$ are just special cases of (iii)$_2$.

Given such a collection of $\sigma_k(x)$, we would like to combine them into a 
generating series---{\it It is such series that we will identify as 
2-functions below}---
\begin{equation}
\eqlabel{generating}
L_D(x;z) = \sum_{k=1}^\infty \frac{\sigma_k(x)}{k^2} z^k
\end{equation}
and study possible analytic properties of $L_D(x;z)$ as a function of $z$.

It is in fact not hard to find $\sigma_k(x)$ that satisfy these conditions,
and these solutions can also be lifted modulo higher powers of $p$. The idea
is the following: If $k=p$ is prime, and $P'(x^p)\not\equiv 0 \bmod p$, we may 
use Newton's formula and define
\begin{equation}
\eqlabel{newton}
\sigma_p(x) = x^p - \frac{P(x^p)}{P'(x^p)}
\end{equation}
which satisfies \eqref{i2} after expansion in $p$. Moreover, iteration of 
\eqref{newton} leads to higher-order solutions. (Of course, this solution is
not unique. Also note that it is necessary in general that $p$ be unramified for 
this formula to make sense. A more conceptual explanation is subsumed in the 
technical part of the paper.) For general $k$, we may define $\sigma_k(x)$
recursively by similar formulas, assuming $\sigma_{k/p}(x)$ has been defined
for all $p$ dividing $k$.

The crux however, is that this solution is far from unique (any modification
of \eqref{newton} by a multiple of $p^2$ is allowed), and it is far from obvious 
that the generating function \eqref{generating} will be anything but a formal 
power series. Therefore, a more meaningful question is whether there is a
choice of the $\sigma_k(x)$ such that $L_D(x;z)$ will have some nice analytic 
properties. 

One extreme case is when $x\in \Z$, for we may then simply take $\sigma_k(x)=x$
for all $k$! Then $L_D(x;z)=x\cdot {\rm Li}_2(z)$, where
\begin{equation}
{\rm Li}_2(z) = \sum_{k=1}^\infty \frac{1}{k^2} z^k
\end{equation}
is the series defining the ordinary di-logarithm. 

Another simple case is when $x=\zeta$ is a root of unity, where we may take
$L_D(\zeta;z)={\rm Li}_2(\zeta z)$. Given this, the Kronecker-Weber theorem 
will provide a natural solution to our problem for any $x$ such that $K=\Q(x)$
has abelian Galois group over $\Q$ (see section \ref{abelian}). 
As a mathematical problem, the question 
then becomes non-trivial when ${\rm Gal}(K/\Q)$ is {\it non-abelian}.

Not surprisingly, the difficulties with finding a natural simple solution in
the general case (see section \ref{basis}), can be traced back to the fact that 
there is no natural global lift of the Frobenius endomorphism at each prime 
of $K$, and that moreover, these Frobenius endomorphisms to not commute amongst
each other.

Without the physics, of course, mathematics knows how to circumvent these 
difficulties. If our goal is to form an analytic function that encodes the global
behaviour (over all primes) of the Galois group, we may pick a finite-dimensional
representation $\rho:{\rm Gal}(K/\Q)\to {\rm End}(V)$, for some complex
vector space $V$, and consider the associated (Artin) $L$-function $L(s;\rho)$,
which is built out of characteristic polynomials of the representation.
These $L$-functions are of course much studied.

What the attachment of strings suggests is that, certainly up to $s=2$, there
{\it exists a different way} of producing an interesting analytic function
involving similar data. According to the ideas of section \ref{physmotiv}, the
physical setup will involve a vector space acted upon by the Galois group (this might 
not quite be a representation, but should be closely related). Moreover, to
the extent that the physical setup has a geometric origin\footnote{We do not
explain the connection between physics and geometry in any detail in this paper,
referring instead to \cite{opening,normal,svw1}}, it will produce a $2$-function 
with the requisite properties.

Our main evidence for this claim is simply that the geometric setup (see
section \ref{geosit}) 
in some sense {\it already provides a solution} to the above problem. Indeed, we 
will prove in section \ref{proof} that certain Hodge theoretic invariants 
associated to algebraic cycles on Calabi-Yau 3-folds, when expanded in the 
appropriate coordinates, satisfy congruence relations of exactly the above
type. Moreover, these functions by construction have sensible analytic
properties. Therefore, the above conditions are not impossible to
satisfy.

From the abstract point of view, the geometric origin of the solutions is
not entirely satisfactory. For one thing, the number field $K$ is dictated
by the geometry, so we cannot produce a solution for arbitrary choice of $K$.
This also means that the expansion coefficients carry information that it
not intrinsic to the arithmetic of $K$.

We hope that a better understanding of the physics will allow us to lift
these limitations. We find it particularly encouraging that although we 
do not fully understand the physics implementation of the Galois symmetries, 
we see no way that physics cares about the distinction between the Galois
group being Abelian or non-Abelian. So conceivably, a better understanding 
of the physics could lead to a solution also in the non-Abelian case.
(Again, the alternative would be that physics does care about the nature
of the Galois group, which would be at least as interesting.)

The other main result of our paper is a piece of evidence for the idea that
among all possible strengthenings of the congruences (i), (ii), (iii),
\ie, replacing $2\mapsto s$ in (i)$_2$, (ii)$_2$, and (iii)$_2$, the initial 
non-trivial choice, $s=2$ is distinguished by the existence of the framing
operation. We now turn to these formal developments.
 
\section{2-functions and their framing}
\label{framing}

\subsection{Definitions and Results}
\label{definitions}

\paragraph{Preliminaries.}
Let $K$ be a finite field extension of $\Q$.\footnote{We do not assume that $K$
is Galois over $\Q$.} We denote by $\G$ the ring of integers in $K$,
and by $\G_D$ the ring of elements of $K$ that are integral outside the discriminant
of $K/\Q$. For a rational prime $p$, unramified in the extension $K/\Q$, we consider the 
$p$-adic completion of $\G$,
\begin{equation}
\eqlabel{unconv}
\G_p = \varprojlim_n \G/(p^n \G)
\end{equation}
Unless $p$ is inert (\ie, unless $(p)=p\G$ is a prime ideal in $\G$), $\G_p$ is not an 
integral domain. ($\G/(p)$ is not a field.) In general, we have a factorization
\begin{equation}
(p) = \prod_{i=1}^r \fp_i
\end{equation}
into $r$ distinct prime ideals of $\G$ (we are assuming that $p$ is unramified), 
and $\G/(p)=\prod_i \G/\fp_i$. In fact, by the Chinese Remainder Theorem, for all 
$n\ge 1$, there is a canonical isomorphism
\begin{equation}
\G/(p^n) \cong \prod_{i=1}^r \G/\fp_i^n
\end{equation}
and so 
\begin{equation}
\eqlabel{andso}
\G_p = \varprojlim_n \prod_{i=1}^r \G/ \fp_i^n = 
\prod_{i=1}^r \varprojlim \G/\fp_i^n=
\prod_{i=1}^r \G_{\fp_i}
\end{equation}
where $\G_{\fp_i}$ are the (more) standard rings of $\fp_i$-adic integers. The $\G_{\fp_i}$
are integral domains and their field of fractions, 
$K_{\fp_i} = (\G_{\fp_i}\setminus\{0\})^{-1} \G_{\fp_i}$, is the $\fp_i$-adic 
completion of $K$. It is also true (though perhaps less canonical) that
\begin{equation}
K_{\fp_i} = (\Z\setminus\{0\})^{-1} \G_{\fp_i}
\end{equation}
So in view of \eqref{andso}, we define
\begin{equation}
\eqlabel{point}
K_p := (\Z\setminus\{0\})^{-1} \G_p 
\end{equation}
We have
\begin{lemma}
\label{product}
\begin{equation}
\eqlabel{diagonally}
K_p= \prod_{i=1}^r K_{\fp_i}
\end{equation}
\end{lemma}
\qed

The point of defining $K_p$ via \eqref{andso} (instead of directly as a product
of fields) is that it makes the following construction of the Frobenius endomorphism 
more natural (to us). In particular, it is independent of Galois theory in 
$K_{\fp_i}$, allowing for several generalizations of our construction (and, in particular,
of Theorem \ref{framtheo}).

We also note that $K$ is canonically embedded in $K_p$ (namely, diagonally in
the product \eqref{diagonally}).

\paragraph{Frobenius.} In $\G/(p)$, we have the endomorphism
\begin{equation}
\label{basfrob}
\frob_p : \G/(p)\to \G/(p)\,,\qquad x\mapsto x^p
\end{equation}
which under the isomorphism $\G/(p)\cong \prod_i \G/\fp_i$ is identified with the
standard Frobenius element in the Galois group of each local field extension 
$(\G/\fp_i)/(\Z/(p))$. By Hensel's Lemma (or more simply, Newton's formula,
\eqref{newton}), $\frob_p$ has a canonical lift to $\G_p$ of \eqref{andso},
and can then be extended to $K_p$ by noting that $\frob_p|_\Z=\id$. We denote
these by the same symbol, and note that the result of the definition coincides
with the Frobenius element $\frob_{\fp_i/p}$ in each local extension 
$K_{\fp_i}/\Q_p$.

\paragraph{Power series.} Letting $z$ be a (formal) independent variable,
we consider the ring of formal power series $K[[z]]$, with obvious 
embeddings $K[[z]]\hookrightarrow K_p[[z]]$, $K[[z]]\hookrightarrow K_{\fp_i}[[z]]$ 
and a (compatible) morphism 
$K_p[[z]]\to K_{\fp_i}[[z]]$ for each prime $\fp_i$ over $p\in\Z$.
Given $V\in K[[z]]$, we denote its image in $K_p[[z]]$ and $K_{\fp_i}[[z]]$
by $V_p$ and $V_{\fp_i}$, respectively. We use similar notation for integral
coefficients. For instance, $\G_D[[z]]$ is the ring of formal power series
with coefficients that are integral outside the discriminant.
We extend $\frob_p$ to an endomorphism
of $K_p[[z]]$ by declaring
$$
\frob_p(z) = z^p
$$
We also introduce the logarithmic derivative
\begin{equation}
\begin{split}
\delta := \delta_z : K[[z]]\to & z K[[z]] \subset K[[z]] \\
\delta_z(V) & := z \frac{dV}{dz}
\end{split}
\end{equation}
and its (partial) inverse
\begin{equation}
\dint : z K[[z]]\to z K[[z]]
\end{equation}
Explicitly, if
$$
V = \sum_{k=1}^\infty a_k z^k
$$
then
$$
\delta V = \sum_{k=1}^\infty k a_k z^k\,,
\qquad
\dint V = \sum_{k=1}^\infty \frac{a_k}{k} z^k
$$
An important observation is
\begin{lemma}
\label{lemma2}
\begin{gather*}
\delta \circ \frob_p = p \frob_p \circ \delta \\
\frac{1}{p} \dint \circ \frob_p = \frob_p \circ\dint
\end{gather*}
\end{lemma}
\qed

\noindent
In particular, we have $\delta (\G[[z]]) \subset \G[[z]]$, but 
$\dint$ does not preserve integrality in general.

\paragraph{$s$-functions.} Let $s$ be a non-negative integer.
A formal power series $V\in z K[[z]]$ is called an {\it $s$-function with 
coefficients in $K$} if for every unramified prime $p\nmid D(K/\Q)$, we have
\begin{equation}
\eqlabel{central}
\frac{1}{p^s}{\rm Frob}_p V_p - V_p \in z \G_p[[z]]
\end{equation}

\begin{lemma}
\label{deltaint}
If $s>0$ and $V\in K[[z]]$ is an $s$-function, then $\delta V$ is an 
$(s-1)$-function.
\end{lemma}
\begin{proof}By Lemma \ref{lemma2},
$$
\frac{1}{p^{s-1}} \frob_p \delta V_p - \delta V_p 
= \delta\bigl(\frac{1}{p^s} \frob_p V_p - V_p\bigr)
\in\delta \bigl(z \G_p[[z]]\bigr) \subset z \G_p[[z]]
$$
\end{proof}

We will sometimes find it convenient to verify the $s$-function property at
the level of the coefficients of the power series. (The following lemma
formalizes conditions (i)$_s$, (ii)$_s$, and (iii)$_s$ from the introduction.)
\begin{lemma}
\label{lemma4}
If $V\in z K[[z]]$ is an $s$-function, then $\delta^s V\in \G_D[[z]]$, so writing
\begin{equation}
\eqlabel{expansion}
V = \sum_{k=1}^\infty \frac{a_k}{k^s} z^k
\end{equation}
we have that all $a_k\in \G_D$. Moreover, letting for fixed $k$ and $p$ prime, 
$\alpha={\rm ord}_p(k)$, we have
\begin{equation}
\eqlabel{gives}
\frob_p (a_{k/p}) - a_{k} = 0 \bmod p^{s\alpha} \G_p
\end{equation}
(with the understanding that $a_{k/p}=0$ if $p\nmid k$).
Conversely, if this condition holds for every $k$ and unramified prime $p$, 
then $V$ is an $s$-function.
\end{lemma}
\begin{proof} Plugging \eqref{expansion} into \eqref{central}, the coefficient
of $z^{k}$ gives the condition
$$
\frac{1}{p^s} \frob_p \frac{a_{k/p}}{(k/p)^s} - \frac{a_{k}}{k^s} \in \G_p
$$
Multiplying with $p^{s\alpha}$, and given that $\frac{p^\alpha}{k}\in \G_p$,
this is equivalent to \eqref{gives}.
\end{proof}
Finally, we note that thanks to Lemma \ref{product}, we can equivalently characterize
$s$-functions by the behaviour at the primes of $K$.
\begin{lemma}
$V\in K[[z]]$ is an $s$-function if and only if for every prime ideal $\fp$ of
$\G$ that is not a branch point of $\spec(\G)\to \spec(\Z)$, we have
$$
\frac{1}{p^s}
\frob_{\fp/p} V_\fp - V_\fp \in z \G_\fp[[z]]
$$
where $(p)=\fp\cap \Z$.
\end{lemma}
\qed
\begin{remark}
We have here defined $s$-functions as power series without constant term.
In applications, they are often accompagnied by a non-zero constant term, 
but the properties of that term depend on the context. For instance,
the constant term might take values in a transcendental extension of $K$,
with a rather different action of the Galois group. For a different
example, with an algebraic constant term, see section \ref{multi}.
We also emphasize explicitly that we do not impose
any condition at the ramified primes, although we suspect that it would
be interesting to do so.
\end{remark}

\paragraph{1-functions} We will now show that 1-functions (\ie, $s$-functions
with $s=1$) are simply linear combinations of ordinary logarithms. Specifically, 
we claim that for any $1$-function $V\in zK[[z]]$ there exists a sequence 
$(b_d)\subset \G_D$ such that (as formal power series)
\begin{equation}
\eqlabel{logcomb}
V = -\sum_{d=1}^\infty \log (1-b_d z^d)
\end{equation}
Conversely, any power series of this form is a $1$-function. As a result, we
obtain a version of the celebrated ``Dwork integrality lemma''
\begin{proposition}
\label{dworklemma}
Let $V\in z K[[z]]$ and $Y\in 1+zK[[z]]$ be related by
$V=\log Y$, $Y=\exp (V)$. Then the following are equivalent:\\
(i) $V$ is a $1$-function 
\\
(ii) For every (unramified) prime $p$,
\begin{equation}
\eqlabel{dwork}
\frac{\frob_p Y_p}{(Y_p)^p} \in 1+ z p \G_p[[z]].
\end{equation}
\\
(iii) $Y\in 1+ z \G_D[[z]]$
\end{proposition}
\begin{proof}
We begin with \eqref{logcomb}. Writing 
$$
V = \sum_{d=1}^\infty \frac{a_d}{d} z^d = \sum_{d,k=1}^\infty \frac{(b_d z^d)^k}{k}
$$
and comparing coefficients, we obtain
\begin{equation}
\eqlabel{relation}
\frac{a_d}{d} = \sum_{k|d} \frac{(b_{d/k})^k}{k} 
\end{equation}
By Lemma \ref{lemma4}, what we have to show is that $b_d\in \G_D$ for all $d$ 
iff $\frac{1}{d}(\frob_p a_{d/p} - a_{d})\in \G_p$ for all 
$d$, and prime $p$. The key observation is that if $b_{d/k}\in\G_D$ then, by
Euler's theorem, for all $p$,
\begin{equation}
\eqlabel{kye}
(b_{d/k})^{k p} = \frob_p (b_{d/k})^{k} \bmod p^{{\rm ord}_p(k)+1}\G_p
\end{equation}
Therefore, assuming $b_d\in \G_D$ for all $d$, we have
\begin{equation}
\begin{split}
\frac{\frob_p a_{d/p}}{d} =\frac{1}{p}\frob_p\Bigl(\frac{a_{d/p}}{d/p}\Bigr) &= 
\sum_{k|\frac{d}{p}} \frac{\frob_p(b_{d/kp})^k }{kp}  \\
&=
\sum_{k|\frac{d}{p}} \frac{(b_{d/kp})^{kp}}{kp} \bmod \G_p \\
&=
\sum_{\topa{k|d}{p|k}} \frac{(b_{d/k})^{k}}{k}
+
\sum_{\topa{k|d}{p\nmid k}}
\frac{(b_{d/k})^{k}}{k} \bmod \G_p \\
&= \frac{a_{d}}{d} \bmod \G_p
\end{split}
\end{equation}
For the converse, we first note that by eq.\ \eqref{relation}, $b_1=a_1\in\G_D$
in any case. Then, by way of induction, we assume that for some $d>1$,
we have established $b_{d/k}\in\G_D$ for all $k|d$. For any $p$, with
$\alpha={\rm ord}_p(k)$, we have from \eqref{relation}
\begin{equation}
\begin{split}
\frac{a_d}{d} &= \sum_{\topa{k|d}{p\nmid k}} \frac{1}{k}
\sum_{i=0}^\alpha \frac{(b_{d/kp^i})^{kp^i}}{p^i} \\
&= \sum_{\topa{k|d}{p\nmid k}}\frac{(b_{d/k})^k}{k} + 
\sum_{\topa{k|d}{p\nmid k}} \sum_{i=0}^{\alpha-1} 
\frac{(b_{d/k p^{i+1}})^{k p^{i+1}}}{k p^{i+1}}
\\
&=
\sum_{\topa{k|d}{p\nmid k}} \frac{(b_{d/k})^k}{k}
+ 
\sum_{\topa{k|d}{p\nmid k}} \sum_{i=0}^{\alpha-1} 
\frac{\frob_p \bigl(b_{d/k p^{i+1}}\bigr)^{k p^i}}{k p^{i+1}} \bmod\G_p \\
&=
b_d + \frac{\frob_p a_{d/p}}{d} \bmod\G_p
\end{split}
\end{equation}
Therefore, $\frac{1}{d}(\frob_p a_{d/p} - a_{d})\in \G_p$ implies 
$b_d\in \G_p$ for all $p\nmid D$.

Given this, (i) immediately implies
\begin{equation}
\eqlabel{asin}
Y = \prod_{d=1}^\infty (1-b_d z^d)^{-1} \in 1+ z \G_D[[z]]
\end{equation}
\ie, (iii). Given $Y\in 1+z\G_D[[z]]$, it can be factored as in \eqref{asin}, with
$b_d\in\G_D$. Then
\begin{equation}
\frac{\frob_p Y_p}{(Y_p)^p} = \prod_d \frac{(1-b_d z^d)^p}{1-\frob_p b_d z^{dp}}
= \prod_d\frac{1-(b_d)^p z^{dp}}{1-\frob_p (b_d) z^{dp}} = 1\bmod z p \G_p
\end{equation}
implying (ii). Finally, given (ii), taking the logarithm on the two sides, and
using $\log (1+p z \G_p[[z]])\subset p z\G_p[[z]]$ implies (i).
\end{proof}

\paragraph{Framing of 2-functions}

In \cite{svw1}, we motivated framing as an ambiguity in the choice of variables
in which to write our formal power series. Given a $1$-function $V\in z K[[z]]$, 
we can write $Y=\exp(V)\in 1+z\G_D[[z]]$ as a series in $z$, or as a series in 
$z_f = z (-Y)^f\in (-1)^f z + z\G_D[[z]]$, for any integer $f$. The resulting series,
$Y_f\in 1+z_f \G_D[[z_f]]$, will also have integral coefficients, and define 
a ``framed'' $1$-function $V_f = \log Y_f$. Clearly, these ``framing transformations''
preserving integrality are generated by $f=1$, and can be extended to include 
$Y\mapsto Y^{-1}$. Our main theorem says that if $V$ ``comes from'' (in the 
sense of being the logarithmic derivative of) a $2$-function, then so do all its 
framed versions. We first state a somewhat more special result, and return to
the general case in section \ref{multi}.

\begin{theorem}
\label{framtheo}
Let $W\in z K[[z]]$ be a $2$-function. Define $Y=\exp(-\delta W)$ and $\tilde Y(\tilde z)$
via the inverse series $\tilde z = - z Y(z)$, $z=-\tilde z\tilde Y(\tilde z)$. Then
$\tilde W = -\tilde\dint\log \tilde Y(\tilde z)\in \tilde z K[[\tilde z]]$ 
is also a 2-function (where $\tilde\dint$ is the logarithmic integral w.r.t.\
$\tilde z$).
\end{theorem}

We prove this theorem in section \ref{framingproof}. It will then become clear that the
minus sign in the relations between $z$, $\tilde z$, $Y$, $\tilde Y$ is important for
preserving integrality at $p=2$. (Whereas the sign in the relation between $Y$ and $W$ is
conventional.) In the rest of this section, we discuss some examples and ask questions
about possible further theoretical developments.

\subsection{Bases of \texorpdfstring{$s$}{s}-functions}
\label{basis}

Let us denote by $\cals_K\subset z K[[z]]$ the set of $s$-functions with coefficients 
in a fixed number field $K$. One sees immediately that $\cals_K$ is a free module over 
$\Z[\frac 1D]$, where $D$ is the discriminant of $K/\Q$. We view it as an important 
challenge, and
especially for $s=2$, to characterize a submodule of $s$-functions by suitable
algebraic or analytic properties, and a class of distinguished generators for
this submodule. For the following considerations, we endow $\cals_K$ with the 
topology of formal power series in one variable, with neighborhood basis 
$\cals_{K,l}:=(z^l)K[[z]]\cap\cals_K$ for $l=1,2,\ldots$.

\begin{lemma}
(i) If $V(z) = \sum \frac{a_k}{k^s} z^k \in \cals_K$ is an $s$-function, then for 
$l>1$ ${\rm Sh}_l(V)(z) := V(z^l)=\sum \frac{a_k}{k^s} z^{lk}\in \cals_{K,l}$ is 
also an $s$-function.
\\
(ii) $\cals_{K}=\cals_{K,1}$ and $\cals_{K,1}/\cals_{K,2}\cong \calo_D$ is a
free module over $\Z[\frac 1D]$ of rank $d=[K:\Q]$. In fact, 
$\cals_{K,l}/\cals_{K,l+1}\cong\calo_D$
for all $l$.
\\
(iii) If $\{V_1, V_2,\ldots , V_d\}\subset \cals_K$ is a set of $s$-functions 
whose image in
$\cals_{K,1}/\cals_{K,2}$ generates $\calo_D$, then (the image of)
${\rm Sh}_l\{V_1,\ldots,V_d\}$ generates $\cals_{K,l}/\cals_{K,l+1}$, and
\begin{equation}
\cup_{l=1}^\infty {\rm Sh}_l\{V_1,\ldots,V_d\}
\end{equation}
is a (Schauder) basis of $\cals_K$ in the $z$-adic topology of formal power 
series.
\end{lemma}
\begin{proof}
(i) is obvious. (ii) follows from the fact that the leading coefficient of
any $s$-function is in $\calo_D$ (this was noted, e.g., in Lemma \ref{lemma4}).
To verify (iii) one may show recursively that for any $L=1,2,\ldots$
\begin{equation}
\cals_K - \langle \cup_{l=1}^L {\rm Sh}_d\{V_1,\ldots,V_d\} \rangle_{\Z[\frac 1D]}
\subset (z^L)K[[z]]
\end{equation}
\end{proof}
In is natural to call such a set $\{V_1,\ldots, V_d\}$ that generates $\cals_K$ 
over $\Z[\frac 1D]$ and under ${\rm Sh}_l$ a ``basis of $s$-functions with
coefficients in $K$''. To construct such a basis, in view of the Lemma, it is
enough to show that for every algebraic integer $x\in\calo_{\Qbar}$,
there exists an $s$-function with coefficients in $K=\Q(x)$ and leading
coefficient $a_1=x$, as in the Introduction. Indeed, the congruences 
\eqref{gives} relate all coefficients with the Galois orbit of $a_1$ modulo 
$\calo_D$, so that (if a solution to the congruences exists, which we will 
show momentarily) the coefficients will all be in $K$. Note that this is 
true even if $K$ is not Galois over $\Q$ since all 
local extensions are. As preliminary restrictions on the class of allowed functions, 
we will call such an $s$-function $V\in zK[[z]]$ {\it algebraic} if 
$Y:=\exp\bigl(-\delta^{s-1} V\bigr)$ 
is the series expansion of an algebraic function of $z$ around $0$. We call an
$s$-function {\it locally analytic} if (for some embedding $K\hookrightarrow \C$)
it converges (in the complex topology) in a finite neighborhood of
the origin, and we say that $V$ is {\it analytic} if it can be analytically continued
to a dense subset of the complex plane. Clearly, {\it algebraic} $\Rightarrow$
{\it analytic}. Moreover,

\begin{lemma}
For every algebraic integer $x\in\C$ there exists an $s$-function
$V\in xz + z^2 K[[z]]$ that is locally analytic.
\end{lemma}
\begin{proof}
By Lemma \ref{lemma4}, we need to find a convergent power series
$V=\sum \frac{a_k}{k^s} z^k$ with coefficients $a_k$ that satisfy for all 
unramified $p|k$ the condition
\begin{equation}
\eqlabel{recurr}
\frob_p(a_{k/p}) - a_k = 0 \bmod p^{s\alpha}\calo_p
\end{equation}
To this end, given $a_1:=x$, we determine $a_k\in K=\Q(x)$ for $k>1$ (outside 
the discriminant) recursively by (i) fixing for each $p|k$ a lift
 $\frob_p^{(k,s)}:\calo\to\calo$ of Frobenius at $p$ $\bmod p^{s\alpha}$
to $\calo$, and (ii) solving the congruences
\begin{equation}
a_k = \frob_p^{(k,s)}(a_{k/p})  \bmod p^{s\alpha}\calo
\end{equation}
jointly for all $p|k$. (This is possible by the CRT.) Since for every
embedding $K\hookrightarrow \C$, there exists a $B>0$ such that any disk
of radius $B$ contain an element of $\calo$, we can choose $a_k$ such that
$|\frac{a_k}{k^s}|<B$. We put $a_k=0$ when $(D,k)\neq 1$.
Then $V=\sum \frac{a_k}{k^s} z^k$ has radius of convergence 
at least $B$.
\end{proof}
\begin{remark}
This algorithm of course is far from specifying a unique solution to the problem,
and the condition of local 
analyticity is clearly too weak to select a finitely generated submodule of 
$s$-functions, motivating us to seek $s$-functions with stronger analytic
properties. We will next show that when $K=\Q(x)$ is an abelian extension,
there exists a basis of algebraic $s$-functions in the above sense. 
An important consequence 
of Theorem \ref{origin} is that algebraic cycles on Calabi-Yau
three-folds provide a source of $2$-functions that are analytic, and even satisfy 
a differential equation with algebraic coefficients, albeit in a different 
variable $q(z)$, that is related to $z$ by a transcendental ``mirror'' 
transformation which however does not preserve $2$-integrality. This
class includes examples with non-abelian Galois group, but
does not teach us how to specify a basis in general. 

\end{remark}

\subsection{Abelian field extensions}
\label{abelian}

We have already remarked in the introduction that if $\zeta$ is a root of unity, 
then $\frob_p(\zeta) = \zeta^p$ for all $p$, and as a 
consequence
\begin{equation}
{\rm Li}_s(\zeta z) = \sum_{k=1}^\infty \frac{\zeta^k}{k^s} z^k
\end{equation}
is an (analytic) $s$-function for any $s$.

Let us now assume that $K$ is a number field that is Galois over $\Q$ with Galois 
group ${\rm Gal}(K/\Q)$ that is {\it abelian}. Then, by the Kronecker-Weber 
theorem, there exists a root of unity $\zeta$, say primitive of degree $N$, such 
that $K$ is a subextension of $\Q(\zeta)$. By elementary Galois theory, there is 
an (abelian) subgroup $\Gamma\subset{\rm Gal}(\Q(\zeta)/\Q)$ such that $K = 
\Q(\zeta)^\Gamma$, and ${\rm Gal}(K/\Q)={\rm Gal}(\Q(\zeta)/\Q)/\Gamma$.

Then for any algebraic integer $x\in\calo_K$, there are rational numbers $c_i\in\Q$
such that 
\begin{equation}
\eqlabel{xexp}
x = \sum_{i=0}^{N-1} c_i \zeta^i
\end{equation}
and for every $(p,N)=1$
\begin{equation}
\frob_p(x) = \sum_{i=0}^{N-1} c_i \zeta^{pi} \in\calo
\end{equation}
In particular, $c_i = c_{p^{-1}i\bmod N}$ whenever $\frob_p\in\Gamma$.
Clearly then,
\begin{equation}
L_D(x;z):=\sum_{i=0}^{N-1} c_i {\rm Li}_s(\zeta^i z) \in xz+z^2 K[[z]]
\end{equation}
is an $s$-function with coefficients in $K$ for every $s$. Since
\begin{equation}
\delta^{s-1}L_D(x;z) = - \sum_{i} c_i \log (1- \zeta^i z)
\end{equation}
we see that $L_D(x;z)$ is {\it algebraic}. As a consequence
\begin{theorem}
If $K=\Q(x)$ is an abelian extension of $\Q$, there exists an algebraic basis 
of $s$-functions with coefficients in $K$.
\end{theorem}
As a physical example from the introduction \ref{physmotiv}, 
consider $x$ a root of $x^3+x^2-2x-1$. We have 
$K=\Q(x)=\Q(\zeta)^\Gamma$, where $\zeta$ is a primitive $7$-th root
of unity, and $\Gamma=\Z/2$ whose non-trivial element acts by $\zeta
\mapsto \zeta^{-1}$. Namely, $x=\zeta+\zeta^{-1}$, and
\begin{equation}
\delta^{s-1} L_D(x;z) = -\log(1-x z+z^2)
\end{equation}

\subsection{When do 2-functions come from 3-functions?}

We reiterate here a few comments from \cite{svw1} concerning the special status of
$2$-functions. First of all, given our results on framing of $2$-functions
(Theorem \ref{framtheo} and its generalization, Theorem \ref{ffram}), it seems 
natural to ask whether starting from an $s$-function with $s>2$ and taking $s-1$ 
logarithmic derivatives, framing \`a la Thm.\ \ref{framtheo} might produce 
a 1-function that can be integrated back to 
an $s'$-function with $s'>2$. In general, this is not the case (the proof that we
give below makes it plain why one should not expect it, and one easily produces
counterexamples). However, there can be special cases in which it is, and so one 
comes to ask which pairs of $s$-, $s'$-functions are related by framing in this 
fashion. 

The simplest example for this phenomenon (with $s'=3$) comes from the ordinary 
polylogarithms, ${\rm Li}_s$, which are of course $s$-functions with rational 
coefficients for any $s$. For $f\in\Z$, cmp.\ Theorem \ref{ffram}, we solve
\begin{equation}
\eqlabel{lisinv}
z_f = z (-\exp({\rm Li}_1(z)))^f = \frac{z}{(z-1)^f}
\end{equation}
for $z$,
\begin{equation}
z = (-1)^f z_f Y_f
\end{equation}
with $Y_f\in 1+z_f\Z[[z_f]]$, and claim that
\begin{equation}
F_f = {\dint}\!\!\!\dint \log Y_f
\end{equation}
is a 3-function for all $f$, except perhaps at $p=2$ and $3$. Namely, writing 
\begin{equation}
\eqlabel{fcoeffs}
F_f = \sum_{d=1}^\infty N_d^{(f)} {\rm Li}_3(z_f^d)
\end{equation}
we claim that the $N_d^{(f)}$ are integers (after multiplication by a power
of $6$) for all $d$ and $f$. (See Table \ref{23} for some examples; it seems 
that in fact $6 N_d^{(f)}/f\in\Z$.)
\begin{table}
\begin{center}
\begin{tabular}{l|l|l|l|l|}
 $f$ & 2 & 3 & 4 & 5 \\\hline 
1 & -2 & 3 & -4 & 5 \\ 
2 & 1  & $\frac{3}{2}$ & 4 & 5 \\ 
3 & $-\frac{2}{3}$ & 3 & -8 &  $\frac{50}3$ \\
4 & 1 & $\frac{15}2$ & 28 &  75 \\
5 & -2 & 24 & -124 &  425 \\
6 &  $\frac{13}{3}$ &  $\frac{171}2$ & 624 & $\frac{8240}3$ \\
7 & -10 & 339 & -3452 & 19605 \\
\end{tabular}
\caption{Framed sequence $N_d^{(f)}$ from \eqref{fcoeffs} for various $d$, $f$.}
\label{23}
\end{center}
\end{table}

To prove our claim, we note that the explicit solution of \eqref{lisinv} is
given by
\begin{equation}
V_f = -\log Y_f = (-1)^f \sum_{k=1}^\infty \frac 1k\binom{k f}{k} z_f^k
\end{equation}
so that in view of Lemma \ref{lemma4}, the statement is equivalent to 
\begin{equation}
\eqlabel{jaka}
\binom{p k f}{pk} \equiv \binom{k f}{k}  \bmod p^{3(\alpha+1)}
\end{equation}
for all $k,f$ and primes $p>3$ (and as before $\alpha={\rm ord}_p(k)$). 
The congruence \eqref{jaka} now follows from the generalization of the
classical Wolstenholme theorem that is known to experts  \cite{granville}
as the Jacobsthal-Kazandzidis congruence \cite{jacobsthal}.%
\footnote{This congruence is usually stated as
\begin{equation}\eqlabel{usual}
\binom{pn}{pm} \equiv \binom{n}{m} \bmod p^q
\end{equation}
where $q$ is the power of $p$ dividing $p^3 m n(n-m)$. Plugged into \eqref{jaka}
it shows that $F_f$ can be the derivative of a local $s'$-function with $s'>3$ 
for special values of $f$ and $p$. But even \eqref{usual} is not in general
optimal in $q$.}

Meanwhile, further explicit examples of algebraic 3-functions with rational 
coefficients have appeared in \cite{gasu} as solutions of so-called 
extremal A-polynomials of certain knots, and all framings of these 
3-functions are also (algebraic) 3-functions. In this case, the integrality
can be seen to follow from the relation with quiver representation
theory \cite{gasu2}. (Alternatively, it has been suggested that the integrality
can be proved by clarifying the relation between the K-theoretic
``quantizability condition'' of the A-polynomial \cite{gusu} and the
K-theoretic interpretation of $2$-functions that we have given in \cite{svw1}.)

Given that the polylogarithms are arguably the simplest $s$-functions, it seems 
unlikely that framings of $s$-functions can be $s'$-functions with $\min(s,s')>3$
(outside a finite number of primes). We would also be interested to learn
about any other examples of algebraic $3$-functions with rational or algebraic
coefficients.\footnote{Transcendental $3$-functions with algebraic coefficients 
(in real quadratic number fields) have appeared as solutions of certain 
Calabi-Yau-type differential equations studied by Bogner, van Straten et al.\ 
(private communication).}

Another reason for the distinguished status of $2$-functions is that the multi-variable
generalization of framing that we discuss in section \ref{multi} only makes sense
for $2$-functions, although we find it conceivable that $s=3$ could again
harbour some exception.

\section{Proof of Integrality of Framing}
\label{framingproof}

We will prove theorem \ref{framtheo} separately for each rational prime 
$p$. In fact, our main calculation goes through with the following slightly 
more general set of coefficients (which we will have occasion to exploit in 
multi-dimensional framing in section \ref{multi}). 
Abusing the notation of section \ref{framing}, we let $\G_p$ 
be a (commutative, unital) ring in which $p$ is not a zero divisor and all 
integers outside of $(p)= p\Z$ are invertible, and we let $K_p\supset \G_p$ be a 
ring extension in which also $p$ is invertible (in other words, $K_p$
contains $\Q$ as a field; we do not need $K_p$ to be complete w.r.t.\ the
$p$-adic norm).

We suppose $\frob_p:K_p\to K_p$ to be a ring homomorphism fixing $\Q\subset K_p$,
and such that for $a\in \G_p$, $\frob_p (a) - a^p \in p\G_p$. We consider
$K_p[[z]]$ ($\supset \G_p[[z]]$) the ring of formal power series with 
coefficients in $K_p$ ($\supset \G_p$). We extend $\frob_p$ to $K_p[[z]]$
by $z\mapsto z^p$ as usual. A crucial property of this extension is 
that if $X\in\G_p[[z]]$, then
\begin{equation}
\eqlabel{crucial}
X^p - \frob_p X \in p \G_p[[z]]
\end{equation}
(This follows by a two line calculation from the definition $\G_p[[z]]=
\varprojlim \G_p[z]/z^n$.)

Now let $W\in z K_p[[z]]$ be a formal power series satisfying the following
``local 2-function property'' (In this section, we work with a fixed prime $p$,
so we drop the subscript from $W$ etc.)
\begin{equation}
\eqlabel{property}
X := \frac{1}{p^2} \frob_p W - W \in z \G_p[[z]]
\end{equation}
As in lemma \ref{lemma4}, this can be equivalently rewritten in terms 
of the coefficients of $W$ and $X$. With
\begin{align}
\eqlabel{Wdef}
W &= \sum_{k=1}^\infty \frac{a_k}{k^2} z^k \\
\eqlabel{Xdef}
X &= \sum_{k=1}^\infty x_k z^k
\end{align}
we have 
\begin{equation}
\eqlabel{ivsk}
x_k = \frac{a_k - \frob_p(a_{k/p}) }{k^2} \in \G_p
\end{equation}
(with the understanding that $a_{k/p}=0$ if $p\nmid k$).

\begin{lemma}
\label{generallemma}
With these relaxed assumptions on the coefficients, we may still define 
$Y = \exp(-\delta W)$, and solve the
relation
\begin{equation}
\tilde z = - z Y
\end{equation}
for $z$,
\begin{equation}
z = - \tilde z \tilde Y
\end{equation}
Then $\tilde Y\in 1 + \tilde z \G_p[[\tilde z]]$ and 
\begin{equation}
\tilde W := -\dint \log \tilde Y \in \tilde z K_p[[\tilde z]]
\end{equation}
is a 2-function at $p$.
\end{lemma}
Our proof relies only on the manipulation of formal power series, property 
\eqref{crucial}, and elementary $p$-adic estimates.
\begin{proof}
We shall verify that the coefficients of $\tilde W$,
\begin{equation}
\tilde W = \sum \frac{\tilde a_k}{k^2} {\tilde z}^k
\end{equation}
satisfy the congruence
\begin{equation}
\eqlabel{verify}
\tilde a_{pk} = {\rm Frob}_p (\tilde a_{k}) \bmod p^{2(\alpha+1)}\,,\quad 
\text{where $\alpha
={\rm ord}_p(k)$.}
\end{equation}
To this end, we briefly digress to recall the Lagrange inversion formula:
If $f(z)$ is a general kind of formal power series 
without constant term and coefficient of $z$ invertible in its coefficient 
ring (this may sometimes be written as $f(0)=0$, 
$f'(0)\neq 0$), and if $g(\tilde z)$ is the compositional inverse of $f$ (\ie,
$g(f(z))=z$, $f(g(\tilde z))=\tilde z$; a unique such $g$ exists by the assumptions 
on $f$), then for every $k$,
\begin{equation}
\eqlabel{lagrange}
\text{(coefficient of $\tilde z^k$ in $g$)} = \frac{1}{k} \text{(coefficient of $z^{k-1}$
in $(z/f)^k$)}
\end{equation}
This formula is most readily understood with complex coefficients as a consequence
of Cauchy's theorem: If $f$ converges and is suitably analytic in a neighborhood
of $0$, we have
\newcommand{\dd}[1]{\,\ensuremath{d{#1}}}
\begin{equation}
\label{integral}
g(\tilde z) = \oint \frac{f'(z)}{f(z) - \tilde z} \,z \dd z
\end{equation}
for a suitably small contour (and $\frac 1{2\pi\ii}$ included in $\oint$).
Expanding
\begin{equation}
g(\tilde z) = \oint \sum_{k=0}^\infty \frac{f'(z)}{f(z)^{k+1}} {\tilde z}^{k} z \dd z
\end{equation}
and noting that inside a small enough circle, $f$ will only vanish at the origin,
we learn that the $k=0$ term vanishes, while for $k>0$ we may integrate by parts 
to obtain 
\begin{equation}
g(\tilde z) = \sum_{k=1}^\infty \frac{\tilde z^k}{k} \oint \frac{1}{f(z)^k} \dd z
\end{equation}
This shows that \eqref{lagrange} is valid for convergent power series with
complex coefficients, but since the coefficients of $g$ are a priori algebraic
in those of $f$, the formula will be valid for formal power series as well.

Along similar lines, the coefficients of (integer and complex) {\it powers} 
$g^l$ of $g$ can be obtained from the expression
\begin{equation}
g^l(\tilde z) = \oint \frac{f'(z)}{f(z) - \tilde z} \, z^l  \dd z
\end{equation}
and we may also take a derivative at $l=0$ to obtain an expression for the coefficients 
of $\log g(\tilde z)$. (In the analytic approach, one needs to be somewhat careful with
the right choice of contours, but this is irrelevant at the formal algebraic level.)

Applied to our situation, with $\tilde z = - z Y(z)$, $z/\tilde z = -\tilde Y(\tilde z)$ 
the formula reads
\begin{equation}
\label{reads}
\tilde\delta \tilde W(\tilde z) = - \log \tilde Y(\tilde z) 
= -\oint \frac{(z Y(z))'}{z Y(z) + \tilde z}\,
\log\bigl(-\frac{z}{\tilde z}\bigr)\, \dd z
\end{equation}
so that, for $k>0$,
\begin{equation}
\label{or}
\tilde a_k = (-1)^{k-1} \oint \frac{1}{z^k Y(z)^k}\, \dd{\log z}
\end{equation}
From now on, we think of $\oint \cdots \dd{\log z}$ as a formal device for extracting 
the constant term of a power series. In particular, it is unchanged if we replace
$z$ by $z^p$ in the integrand. On the other hand, from outside the $\oint$-sign, 
$\frob_p$ would act only on the coefficients of $Y$, so that in combination, 
we obtain
\begin{equation}
{\rm Frob}_p\tilde a_k = 
(-1)^{k-1} \oint \frac{1}{z^{pk} ({\rm Frob}_p Y)^k}\, \dd{\log z}  \\
\end{equation}
By \eqref{property} and the other definitions, we have
\begin{equation}
\frob_p Y = Y^p \exp(-p\delta X)
\end{equation}
Therefore
\begin{equation}
{\rm Frob}_p\tilde a_k 
= (-1)^{k-1}\oint \frac{1}{z^{pk} Y^{pk}}\, \exp\bigl(pk\,
\delta\! X(z)\bigr) \, \dd{\log z}
\end{equation}
 
Now let's first assume that $(-1)^{kp}=(-1)^k$, which is the case if $p$ is odd,
or $p=2$ and $k$ even. Then
\begin{equation}
\eqlabel{then}
{\rm Frob}_p\tilde a_k - \tilde a_{pk} = (-1)^{k-1} \oint\frac{1}{z^{pk}Y^{pk}}
\,\Bigl(\exp\bigl(pk\,\delta\!X\bigr) - 1\Bigr)\,\dd{\log z}
\end{equation}
Our goal now is to control the order at $p$ of the contribution of each term in the 
expansion,
\begin{equation}
\label{expand}
\exp\bigl(pk\,\delta\!X\bigr)-1 = \sum_{r=1}^\infty \frac{(pk)^r}{r!} 
\bigl(\delta\!X\bigr)^r
\end{equation}
exploiting the fact that all power series involved have coefficients in $\G_p$.
To this end, we use the well-known (or easily checked) estimate,
\begin{equation}
{\rm ord}_p (r!) = \sum_{j=1}^\infty \Bigl\lfloor\frac{r}{p^j}\Bigr\rfloor
\le r \sum_{j=1}^\infty \frac{1}{p^j} - \frac{1}{p-1} = \frac{r-1}{p-1}
\end{equation}
in which equality holds if and only if $r=p^s$ is a prime power. Therefore,
\begin{equation}
\label{left}
{\rm ord}_p\bigl(\frac{(pk)^r}{r!}\bigr) \ge r(\alpha+1)-\frac{r-1}{p-1}
\end{equation}
Now if $p>2$,
\begin{equation}
\frac{r-1}{p-1} \le \frac{r-1}{2}\le r-2
\end{equation}
where the latter inequality holds if in addition $r\ge 3$. In that case then
\begin{equation}
r(\alpha+1)-\frac{r-1}{p-1} \ge 2(\alpha+1)+(r-2)\alpha \ge 2(\alpha+1)
\end{equation}
If $p$ is still odd, but $r=2$, we have $\frac{r-1}{p-1}=\frac{1}{p-1}<1$, 
and since the left-hand side of (\ref{left}) has to be integral, it
can be no less than $2(\alpha+1)$ (in fact, it is equal to that).

If $p=2$, and $\alpha\ge 1$ (\ie, $k$ is even), and also $r\ge 3$, then
\begin{equation}
\begin{split}
r(\alpha+1) -\frac{r-1}{p-1}&= 2(\alpha+1)+(r-2)(\alpha+1) -r+1\\
&\ge 2(\alpha+1)+2(r-2)-r+1
=2(\alpha+1)+r-3\\ &\ge 2(\alpha+1)
\end{split}
\end{equation}
Summarizing, when $p>2$, $\alpha\ge 0$, and $r\ge 2$, or when $p=2$, $\alpha\ge 1$, 
and $r\ge 3$,
\begin{equation}
\label{conclude}
{\rm ord}_p\bigl(\frac{(pk)^r}{r!}\bigr) \ge 2(\alpha+1)
\end{equation}
Therefore, $\bmod p^{2(\alpha+1)}$, we can ignore the contribution of those terms 
to (\ref{then}), since all power series involved have integral coefficients.

To begin dealing with the remaining terms, we observe that when $r=1$, we may improve 
the manifest order at $p$ by ``integrating by parts'' (in other words, using 
$\oint \delta(\cdots) \dd{\log z}=0$, and that $\delta$ is a derivation)
\begin{equation}
\oint \frac{1}{z^{pk} Y^{pk}} (pk) \,\delta\!X \dd{\log z}= 
\oint \frac{\delta(z Y)}{z^{pk+1}Y^{pk+1}} (pk)^2 X \dd{\log z} 
\end{equation}
which vanishes $\bmod p^{2(\alpha+1)}$ since $X$ still has coefficients in $\G_p$.
(This is the place where the original 2-function property
enters in the crucial way. Note also that this step can only be taken exactly 
once, \ie, it cannot be repeated for $s$-functions with $s>2$.)

When $p=2$, $r=2$ (but still $\alpha\ge 1$), we find ${\rm ord}_2((2k)^2/2)=
2(\alpha+1)-1$, so what we have to show is that 
\begin{equation}
\eqlabel{reduce}
\oint \frac{1}{z^{2k} Y^{2k}} (\delta\! X)^2 \dd{\log z} = 0 \bmod 2
\end{equation}
To see this, using def.\ \eqref{Xdef}, we first reduce $\bmod 2$, 
\begin{equation}
(\delta\! X)^2= \sum_i i^2 x_i^2 z^{2i} = \sum_{i\;{\rm odd}} x_i^2 z^{2i}
\end{equation}
and {\it then} integrate by parts each term in \eqref{reduce}
\begin{equation}
\oint \frac{1}{z^{2k} Y^{2k}} z^{2i} \dd{\log z} = 
\oint \frac{1}{z^{2k} Y^{2k}} \frac{1}{2i} \delta z^{2i} \dd{\log z}
= -\frac{k}{i} \oint 
\frac{\delta(z Y)}{z^{2k+1}Y^{2k+1}} z^{2i} \dd{\log z}
\end{equation}
which vanishes $\bmod 2$ if $\alpha\ge 1$, and $i$ is odd.

Finally, we consider the situation $(-1)^{pk}=-(-1)^k$, which happens when $p=2$, 
and $k$ is odd (\ie, $\alpha=0$). This leads to a sign change in \eqref{then},
so we have to study
\begin{equation}
\exp\bigl(2k\,\delta\!X\bigr)+1 = 2 + \sum_{r=1}^\infty \frac{(2k)^r}{r!} 
\bigl(\delta\!X\bigr)^r
\end{equation}
When $r$ is not a power of $2$, (in particular, $r\ge 3$), we easily see that
\begin{equation}
{\rm ord}_2\bigl(\frac{(2k)^r}{r!}\bigr) \ge 2
\end{equation}
so we can ignore those terms. When $r=2^s$ is a power of 2, we are confronted 
with the fact that ${\rm ord}_2\bigr((2k)^{2^s}/(2^s!)\bigl)=1$. 
So to verify (\ref{verify}) in this case, we remain with showing that
\begin{equation}
\label{remain}
\oint \frac{1}{z^{2k}Y^{2k}} \Bigl(1+\sum_{s=0}^\infty
\frac{(2k)^{2^s}}{2 \cdot (2^s!)} (\delta\!X)^{2^s}\Bigr)  \,,
\end{equation}
which we now know is integral, in fact vanishes $\bmod 2$. Referring back to eq.\ \eqref{Xdef}, 
we find $\bmod 2$,
\begin{equation}
(\delta\! X)^{2^s} = \sum_{i=1}^\infty i^{2^s} x_i^{2^s} z^{i\, 2^s}
\end{equation}
and we can again ignore the terms with $i$ even.
 On the other hand, when $i$ is odd, we have from eq.\ \eqref{ivsk}
\begin{equation}
x_i = a_i \bmod 2
\end{equation}
In fact, since $\delta W = \sum_i \frac{a_i}{i} z^i$ is a 1-function, we have 
for all $s$, for $i$ odd, and $\bmod 2$,
\begin{equation}
(x_i)^{2^s} = (a_i)^{2^s} = (\frob_2)^{s} (a_i) = a_{i\,2^s} \bmod 2
\end{equation}
So what remains of (\ref{remain}) becomes 
\begin{equation}
\label{becomes}
\begin{split}
\oint \frac{1}{z^{2k}Y^{2k}} \Bigl(1+ \sum_{s=0}^\infty  \sum_{i\;{\rm odd}}
a_{i\,2^s} z^{i\, 2^s} \Bigr) \dd{\log z}
&= \oint \frac{1}{z^{2k} Y^{2k}} \bigl(1+\sum_{j=1}^\infty a_j z^j\bigr) \dd{\log z}\\
&= \oint \frac{1}{z^{2k} Y^{2k}} \bigl(1-\delta^2 W\bigr) \dd{\log z} 
\end{split}
\end{equation}
These manipulations were valid $\bmod 2$, but we now claim that the RHS of \eqref{becomes}
in fact vanishes identically in $K_p$ (we're at $p=2$). Indeed,
\begin{equation}
\eqlabel{you}
\begin{split}
\frac{1}{2k} \delta\Bigl(\frac{1}{z^{2k} Y^{2k}}\Bigr) 
&= - \frac{1}{z^{2k} Y^{2k}} - \frac{1}{z^{2k} Y^{2k+1}} \delta Y \\
&= -\frac{1}{z^{2k} Y^{2k}} (1-\delta^2 W)
\end{split}
\end{equation}
Thus we see that the integrand at the end of \eqref{becomes} is in fact a total 
derivative,
and therefore its constant term vanishes. This completes the proof.
\end{proof}

\section{Multi-dimensional Framing}
\label{multi}

Theorem \ref{framtheo} shows that replacing $z$ with $\tilde z = - z 
\exp(-\delta W)$ transforms a $2$-function $W$ into another $2$-function 
$\tilde W$ related to $W$ via 
\begin{equation}
\tilde Y := \exp\bigl(-\tilde\delta \tilde W\bigr) = \exp\bigl(\delta W\bigr)
=: Y^{-1}
\end{equation}
Since replacing $\tilde W$ with $-\tilde W$ clearly also preserves $2$-integrality, 
we learn that $\delta W$ itself integrates to a $2$-function with respect to 
$\tilde z = - z\exp(-\delta W)$. This can be iterated to conclude
\begin{theorem}
\label{ffram}
Let $W\in z K[[z]]$ be a $2$-function, $Y:=\exp(-\delta W)$. For integer ``framing 
parameter'' $f\in\Z$, let 
\begin{equation}
\eqlabel{ztozf}
z_f:= z (-Y)^f
\end{equation}
Then $\delta W$, viewed as a formal
power series in $z_f$, is the logarithmic derivative of a $2$-function $W_f
\in z_f K[[z_f]]$.
\end{theorem}
The point is that while the ``elementary'' framing operation studied so far is
involutive, \ie, $\tilde{\tilde z} = z$, $\tilde{\tilde W}=W$, framing in the
sense of Theorem \ref{ffram} defines an action of the group of integers on the
set of $2$-functions with coefficients in $K$. We have $\tilde W = -W_1$, \etc.

More explicitly, framing identifies
\begin{equation}
\delta_f W_f := z_f\frac{dW_f}{dz_f} = z \frac{dW}{dz} =:\delta W
\end{equation}
Given \eqref{ztozf}, we have
\begin{equation}
\frac{z}{z_f} \frac{dz_f}{dz} = 1 - f \delta^2 W
\end{equation}
so that
\begin{equation}
\delta W_f = \frac{z}{z_f} \frac{dz_f}{dz}\,\delta W = \delta W (1-f\delta^2 W)
\end{equation}
Thus, we can write the relation between $W$ and $W_f$ more succinctly as
\begin{equation}
W_f = W - \frac f2 \bigl(\delta W\bigr)^2
\end{equation}
and Theorem \ref{ffram} says that if $W$ is a $2$-function of $z$, then $W_f$ is 
a $2$-function of $z_f$ for any $f\in\Z$.

The generalization to the multi-variable case is now clear: If $z^1,z^2,\ldots,
z^n$ are $n$ independent formal variables, we continue the Frobenius endomorphism
at prime $p$ to the ring of formal power series $K[[z^1,\ldots,z^n]]$ via 
$\frob_p(z^i)=(z^i)^p$ for each $i$. We denote by $(z)K[[z^1,\ldots,z^n]]$ the
maximal ideal generated by the $z^i$ (and $(z)\G[[z^1,\ldots,z^n]]$ that with 
integral coefficient \etc). We say that $V\in (z)K[[z^1,\ldots,z^n]]$ is an 
$s$-function if
\begin{equation}
\frac{1}{p^s} \frob_p V_p - V_p \in (z)\G_p[[z^1,\ldots,z^n]]
\end{equation}
for all $p$ as before, see \eqref{central}.

\newcommand{\dd}[1]{\ensuremath{\operatorname{d}\!{#1}}}

Now let $W\in (z)K[[z^1,\ldots,z^n]]$ be a $2$-function with coefficients in $K$, and 
let $\kappa=(\kappa^{ij})\in \Z^{n^2}$, $\kappa^{ij}=\kappa^{ji}$ be a {\it symmetric 
matrix with rational integer coefficients}.%
\footnote{It appears possible that with some extra care, this can be generalized to 
algebraic integer $\kappa^{ij}$, but we have not studied this question in any detail.}
We then define {\it framing of $W$ with respect to $\kappa$} by the pair of 
formulas
\begin{align}
\label{pair1}
z_\kappa^i &= \sigma_i\, z^i \exp\bigl(-{\textstyle\sum_k} \kappa^{ik} \delta_k W\bigr) \\
\label{pair2}
\delta^{(\kappa)}_i W_\kappa &= \delta_i W
\end{align}
where $\delta_j:= z^j\frac{d}{dz^j}$,
$\delta^{(\kappa)}_j:=z^j_{\kappa}\frac{d}{dz^j_\kappa}$, and $\sigma_i\in\{\pm 1\}$ 
is a sign inserted to guarantee integrality at $p=2$, and determined by the diagonal 
elements of $\kappa$,
\begin{equation}
\eqlabel{diagonal}
\sigma_i := (-1)^{\kappa^{ii}}
\end{equation}
To see that this multi-dimensional 
framing is well defined, we first note that $\sum_k\kappa^{ik}\delta_kW$
is a $1$-function, and hence by Lemma \ref{dworklemma}, 
$z^i_\kappa\in z^i \G_D[[z^1,\ldots,z^n]]$. Moreover,
\begin{equation}
\eqlabel{infact}
\begin{split}
\frac{z^j}{z^i_\kappa} \frac{dz^i_\kappa}{dz^j} &= 
\Delta^i_j - \sum_k\kappa^{ik} \delta_{j}\delta_{k} W \\
&=\Delta^i_j \bmod (z)\G_D[[z^1,\ldots,z^n]]
\end{split}
\end{equation}
where $\Delta^i_j=1$ if $i=j$, and $0$ otherwise is the unit matrix. As a 
consequence, the relation \eqref{pair1} can be inverted to find 
$z^i\in z^i_\kappa\G_D[[z^1_\kappa,\ldots z^n_\kappa]]$.
To see that \eqref{pair2} is integrable with respect to the $z^j_\kappa$,
we observe that as a consequence of \eqref{infact}, the (formal) one-form
\begin{equation}
\begin{split}
\sum_i \delta^{(\kappa)}_i W_\kappa \frac{dz^i_\kappa}{z^i_\kappa}
&= \sum_i \delta_i W \Bigl( \frac{dz^i}{z^i} - 
\sum_{j,k}\kappa^{ik} \delta_j\delta_k W \frac{dz^j}{z^j}\Bigr)
 \\
&= \sum_i \Bigl(\delta_i W - \sum_{j,k}
\kappa^{jk} \delta_j W \delta_i\delta_k W \Bigr)\frac{dz^i}{z^i} \\
&= \delta_i\bigl( W - \frac 12 \sum_{j,k}\kappa^{jk}\delta_j W\delta_k W\Bigr)
\frac{dz^i}{z^i}
\end{split}
\end{equation}
is exact in virtue of the symmetry of $\kappa^{ij}$. In other words, we have,
\begin{equation}
\label{simply}
W_\kappa = W - \frac{1}{2} \sum_{j,k}\kappa^{jk} \delta_j W \delta_k W
\end{equation}
viewed as a formal power series in the variables $(z^i_\kappa)$, obtained by 
inverting \eqref{pair1}. 

The following is obvious from \eqref{pair1}, \eqref{pair2}:
\begin{proposition}
\label{framegroup}
Let $\kappa$ and $\kappa'$ be two symmetric integral matrices. Then
$\bigl((z_\kappa)_{\kappa'}\bigr)^i=z_{\kappa+\kappa'}^i$ and
\begin{equation}
(W_\kappa)_{\kappa'} = W_{\kappa+\kappa'}
\end{equation}
In other words, the group of framing transformation in $n$ variables is the additive
group of symmetric integral $n\times n$ matrices.
\end{proposition}
\qed

It appears to be true that whenever $W\in (z)K[[z^1,\ldots,z^n]]$ is a
$2$-function, then for every symmetric integral matrix $\kappa$, 
$W_\kappa\in (z_\kappa)K[[z^1_\kappa,\ldots,z^n_\kappa]]$ is also a
$2$-function. In view of Proposition \ref{framegroup}, it suffices to
establish this for the generators of the group of framing transformations,
in other words for\\
(i) ``single variable framing'', $\kappa^{ii}=1$ for some $i$, all other
$\kappa^{jk}=0$, and
\\
(ii) ``exchange framing'', $\kappa^{ij}=1=\kappa^{ji}$ for some fixed $i\neq j$,
all other $\kappa^{kl}=0$.
\\
We are able to prove (i) for all primes $p$, and case (ii) for all primes except
$p=2$. The last remaining case appears to depend on an improvement of the estimates 
of Lemma \ref{generallemma} that is as yet missing.

\begin{proposition}
For every $2$-function $W$, and $\kappa$ of type (i), $W_\kappa$ is also a
$2$-function.
\end{proposition}
\begin{proof}
Clearly, it is enough to treat the case $i=1$. Writing 
\begin{equation}
K[[z^1,\ldots,z^n]] = K[[z^2,\ldots,z^n]][[z^1]]
\end{equation}
\etc, we view $W$ as a $2$-function with coefficients in $K[[z^2,\ldots,z^n]]$, 
albeit with {\it in general non-zero constant coefficient}, let us call it 
$a_0\in K[[z^2,\ldots,z^n]]$.

Indeed, for every prime $p$, the pair $K_p[[z^2,\ldots, z^n]]\supset
\G_p[[z^2,\ldots,z^n]]$ satisfies the hypotheses of section
\ref{framingproof}
and we have
\begin{equation}
\frac 1{p^2} \frob_p W_p - W_p \in \G_p[[z^2,\ldots,z^n]][[z^1]]
\end{equation}
Namely $(W-a_0)_p\in z^1K_p[[z^2,\ldots,z^n]]$ is a $2$-function without 
constant coefficient, and
\begin{equation}
\eqlabel{inparticular}
\frac 1{p^2}\frob_p (a_0)_p - (a_0)_p\in\G_p[[z^2,\ldots,z^n]]
\end{equation}
Moreover, for $\kappa^{ii}=1$, all other $\kappa^{jk}=0$, we see that
$z^j_\kappa=z^j$ for $j=2,\ldots,n$, while $z^1_\kappa=\tilde z^1$ in the
notation of Lemma \ref{generallemma} (notice that $\sigma_1=-1$). As a consequence,
\begin{equation}
(W-a_0)_\kappa = -\widetilde{(W-a_0)}
\end{equation}
is a $2$-function with coefficients in $K[[z^2,\ldots,z^n]]$. The claim
follows by adding back the constant coefficient, which is unchanged and therefore
still satisfies \eqref{inparticular}.
\end{proof}

\begin{proposition}
For every $2$-function $W$, and $\kappa$ of type (ii), $W_\kappa$ is a
$2$-function at all odd primes.
\end{proposition}
\begin{proof}
We can assume $i=1$ and $j=2$, and by relaxing the coefficients analogous
to the proof of the previous proposition, we might as well pretend that $n=2$.

Under the substitution
\begin{equation}
\eqlabel{ztow}
z^1 = w^1 w^2\,,\qquad z^2 = \frac{w^1}{w^2}
\end{equation}
the ring of formal power series $K[[z_1,z_2]]$ is identified isomorphically 
with the ring
$\bigl(K[w^2,(w^2)^{-1}][[w^1]]\bigr)_+$ of formal powers series in $w^1$ with 
coefficients that are Laurent polynomials in $w^2$, and the following restrictions
on the $w^{1,2}$-degrees $m_{1,2}$, indicated by the subscript $+$:
\begin{equation}
\eqlabel{degconds}
m_1\ge |m_2|\quad\text{ and }m_1\equiv m_2\bmod 2 \,.
\end{equation} 
These conditions ensure that substituting back,
\begin{equation}
\eqlabel{wtoz}
w^1 = \bigl(z^1z^2\bigr)^{1/2}\,,\qquad w^2 = \Bigl(\frac{z^1}{z^2}\Bigr)^{1/2}
\end{equation}
returns a power series in $z^{1}, z^{2}$. The conditions \eqref{degconds} also 
guarantee that 
formal manipulations in $w^{1,2}$ are equivalent to those in $z^{1,2}$.
For all $p$, the Frobenius endomorphism lifts to $w^1, w^2$ as
\begin{equation}
\eqlabel{froblift}
\frob_p(w^{1,2}) = (w^{1,2})^{p} 
\end{equation}
(This lift is not unique at $p=2$, but this is not the origin of our problems there.) 
The logarithmic derivatives with respect to the $w^{1,2}$ are related
to those w.r.t.\ $z^{1,2}$ as
\begin{equation}
\begin{split}
\gamma_1 &:= w^1\frac{d}{dw^1} = 
z^1 \frac{d}{dz^1} + z^2 \frac{d}{dz^2} =
 \delta_1 + \delta_2\\
\gamma_2 &:= w^2\frac{d}{dw^2} = 
z^1 \frac{d}{dz^1}- z^2 \frac{d}{dz^2} =
\delta_1-\delta_2
\end{split}
\end{equation}
so that framing w.r.t.\ $\kappa$ is diagonal in the $w^{1,2}$: With
\begin{equation}
\eqlabel{halffram}
\begin{split}
w^1_\kappa &= - w^1 \exp\bigl(-\textstyle{\frac 12} \gamma_1 W\bigr) \\
w^2_\kappa &= - w^2 \exp\bigl(\textstyle{\frac 12} \gamma_2 W\bigr)
\end{split}
\end{equation}
we recover
\begin{equation}
\begin{split}
z^1_\kappa = w^1_\kappa w^2_\kappa = w^1 w^2 \exp\bigl(- \textstyle
(\frac 12 \gamma_1 - \frac 12 \gamma_2) W\bigr) = z^1 \exp(-\delta_2 W) \\
z^2_\kappa = \frac{w^1_\kappa}{w^2_\kappa} = \frac{w^1}{w^2} 
\exp\bigl(- \textstyle (\frac 12 \gamma_1 + \frac 12 \gamma_2) W\bigr) = 
z^2 \exp(-\delta_1 W)
\end{split}
\end{equation}
If now $W$ is a $2$-function in $K[[z^1,z^2]]$, eq.\ \eqref{froblift} allows us 
to view it as a $2$-function in $\bigl(K[w^2,(w^2)^{-1}][[w^1]]\bigr)_+$. Thanks to
\eqref{degconds}, we can continue to use the Lagrange formula to invert the transformation
 \eqref{halffram} so
that the proof of Lemma \ref{generallemma} (used as in the previous proposition, 
for each variable separately) still goes through, {\it for all primes $p\neq 2$}. This
shows that $W_\kappa$ is a $2$-function in $\bigl(K[w^2_\kappa,(w^2_\kappa)^{-1}]
[[w^1_\kappa]]\bigr)_+$. Substituting $w^1_\kappa=(z^1_\kappa z^2_\kappa)^{1/2}$,
$w^2_\kappa=(z^1_\kappa/z^2_\kappa)^{1/2}$, we conclude that $W_\kappa$ is a
$2$-function in $K[[z^1_\kappa,z^2_\kappa]]$ at all primes $p\neq 2$.
\end{proof}

\begin{remark}
Because of the $\frac 12$ in the framing \eqref{halffram}
of the $w^{1,2}$-variables, the 
proof of Lemma \ref{generallemma} does not directly apply at $p=2$ for the
two variables separately. A promising line of attack is to use multivariate
Lagrange inversion in the $z^{1,2}$ variables, but we have not been able to
carry this to the end so far.
\end{remark}

\section{2-functions from algebraic cycles on Calabi-Yau threefolds}
\label{geosit}

In the first part of this section, we recall the standard setup of the
B-model of mirror symmetry. Namely, following \cite{guide,delignelimit}, we
describe the variation of Hodge structure attached to a family of complex
Calabi-Yau threefolds, and the special properties of that variation around a point of
maximal degeneration. In particular, we review the interpretation of the 
canonical coordinate (a.k.a.\ the mirror map), as well as the Yukawa coupling, 
as extension classes in the category of mixed Hodge structures. The 
comparison with the $p$-adic analogue of this interpretation is the first
ingredient in the integrability proofs of \cite{ksv,sv1,vadim}. (The
second ingredient is the identification of the limiting behaviour of these
extension classes in the complex and $p$-adic setup, see \cite{vadim}.)

In the second part of this section, we describe, following \cite{normal,
arithmetic}, the extension of the B-model by a family of algebraic cycles 
varying inside the family of threefolds. This includes the extension of the
local system and the relation between the superpotential and the Griffiths
infinitesimal invariant characterizing the extension of Hodge structure.

Finally, we add the assumption that the maximal degeneration of our family
is defined over the integers. This assumption implies that the local period
ring is the ring of power series with rational coefficients. The limit of
the algebraic cycle then is defined over an algebraic number field, which 
leads to an extension of the residue field of the period ring. Our main 
integrality statement is that the superpotential is a 2-function (with 
coefficients in the extended residue field). The statement will be proven 
in the two subsequent sections.

\subsection{Variation of Hodge structure}
\label{VHS}

Let $B$ be a smooth quasi-projective complex curve and let $\pi:Y\to B$ be a 
smooth family of projective Calabi-Yau threefolds parametrized by $B$. We assume (for 
convenience) that the generic member of the family  $Y_b=\pi^{-1}(b)$ ($b$ a 
point in $B$) is simply connected, has middle-dimensional Betti number 
$b_3(Y_b)=4$, and that there is no torsion in cohomology (ever). 

To such a family is associated a polarized, integral variation of pure Hodge 
structure (VHS) $\calh$ of weight $3$ over $B$. The data for the VHS arises as follows.
\\
$(1)$ The local system is the higher direct image $\calh_\Z= R^3\pi_*\Z$ of
the constant sheaf $\Z$ on $Y$. The fibers $(\calh_\Z)_b$ of this local system 
are the middle-dimensional integral cohomology groups 
$H^3(Y_b,\Z)$. Under our assumptions, $\calh_\Z$ is torsion free of rank $4$.
\\
$(2)$ The decreasing Hodge filtration $F^0\supset F^1\supset F^2\supset F^3$ 
on $\calh_\Z\otimes \calo_B=\calh=F^0$ originates in the natural filtration 
on the relative de Rham complex $(\Omega^*(Y/B),d)$. 
The fibers of $F^s$ are $(F^s)_b = \oplus_{s'\ge s} H^{3-s'}(\Omega^{s'}(Y_b))$.
The assumption that $Y_b$ is Calabi-Yau implies that $F^3$ has rank one.
 \\
$(3)$ The anti-symmetric polarization form $\langle\cdot,\cdot\rangle:
\calh_\Z\otimes\calh_\Z\to \Z(-3)$ is induced from the cup-product on cohomology,
and extended linearly to $\calh$.
Here, $\Z(-3)=(2\pi\ii)^{-3} \Z\hookrightarrow \C$ denotes the trivial, constant 
VHS of weight $6$ on $R^6\pi_*\Z\otimes\calo_B$.%

We will denote as usual by $\nabla$ the Gauss-Manin connection on $\calh$ as a vector
bundle, characterized by the property that its horizontal sections are precisely the 
sections of $R^3\pi_*\C=\calh_\C=\calh_\Z\otimes\C$. The connection preserves the
polarization and together with the Hodge filtration enjoys Griffiths 
transversality, that is $\nabla F^s \subset F^{s-1}\otimes\Omega^1(B)$.

Now let us assume that our curve $B$ is embedded into a larger, smooth and projective 
curve $\bar B$, and that our family can be continued to a semi-stable map 
$\bar\pi:\bar Y\to \bar B$. 
Fix a boundary point $a\in\bar B\setminus B$, and restrict to a simply connected
neighborhood $\bar U$ of $a$ in $\bar B$ such that $U=\bar U\setminus\{a\}\subset
B$. We denote the restricted data by the same letters as above. Let $M:(\calh_\Z)_b
\to (\calh_\Z)_b$ be the local monodromy operator of the local system around $a$. 
By the monodromy theorem, $M$ is quasi-unipotent. We assume that $M$ is in fact 
unipotent, and define its logarithm $N=\log M:(\calh_\Q)_b\to (\calh_\Q)_b$, where 
$\calh_\Q=\calh_\Z\otimes \Q$. 

In this situation, the monodromy weight filtration, $W_*$, is the unique increasing 
filtration on $\calh_\Q$ such that $W_{-1}=0$, $W_6=\calh_\Q$, $N W_k\subset W_{k-2}$, 
and that for $k=0,1,2,3$, $N^k$ induces an isomorphism ${\rm Gr}^W_{3+k} \overset{\cong}
{\longrightarrow}{\rm Gr}^W_{3-k}$ between the graded pieces, ${\rm Gr}^W_k=W_k/W_{k-1}$. 

Because of the monodromy, the local system $\calh_\Z$ can not be continued from
$U$ across $a$ to $\bar U$. However, because the monodromy is unipotent, the vector 
bundle $\calh=\calh_\Z\otimes \calo_U$ has a (Deligne) canonical continuation%
$\bar\calh\to\bar U$. The characteristic property 
of the continuation 
is that the flat connection acquires a first order pole at $a$ with residue
(conjugate to) $-N/(2\pi\ii)$. It can be explicitly constructed as follows. One
picks a local coordinate $z$ on $\bar U$ vanishing at the boundary point $a$,
and then introduces on $\calh\to U$ the ``un-twisted'' connection
\begin{equation}
\eqlabel{untwisted}
\nabla^c = \nabla + \frac{N}{2\pi\ii} \frac{dz}{z}
\end{equation}
This connection has no monodromy (the $\nabla^c$-horizontal sections are of the form
$\exp\bigl(-\frac{\log z}{2\pi\ii} N\bigr) g$ for $\nabla$-horizontal section
$g$), which allows continuation of $\calh$ to $\bar U$ as a ``constant'' bundle.
The continued connection is $\bar\nabla = \nabla^c - \frac{N}{2\pi\ii} \frac{dz}{z}$.
This explicit construction will enter later in the definition of the limiting mixed 
Hodge structure, and the comparison with the $p$-adic setup. 

A result of central importance for the present description of the VHS in the neighborhood 
of $a$ is the nilpotent orbit theorem, which guarantees that not only $\calh$, but 
in fact the entire Hodge filtration can be continued across the boundary point.
We will denote it by $\bar F^*$.

Now, the key assumption that makes such a family $Y\to B$ interesting for us
is that the distinguished boundary point $a$ be a {\it point of maximal degeneration}. 
By definition, this means that the local monodromy operator $M$ is unipotent of 
maximal rank $3$. In other words, $(M-{\rm id})^4=0$, but $(M-{\rm id})^3\neq 0$. 
The logarithm $N$ of $M$ is then nilpotent of rank 3.

Under the assumption that the monodromy is maximally unipotent, the monodromy
weight filtration $W_*$ on $\calh_\Q$ pairs up with the Hodge filtration $F^*$ 
on $\calh$ to define a variation of mixed Hodge structure in a punctured neighborhood 
$U$ of $a$ in $B$ as above \cite{delignelimit}, to which we restrict the following 
discussion.\footnote{\label{nilpotent}
We emphasize that we are not here talking about the ``nilpotent
orbit'', which is a different VMHS obtained by extending back the limiting 
Hodge filtration $\bar F^*_a$ as a $\nabla^c$-constant filtration $F^*_{\rm nilp}$
on $\calh$. The theorem is
important, but the nilpotent orbit itself will not play a role in our discussion.}
This mixed Hodge structure is Hodge-Tate, meaning that the pure Hodge
structures 
induced on the even graded pieces ${\rm Gr}^W_{2s}$ are constant of Hodge type $(s,s)$,
while the odd pieces ${\rm Gr}^W_{2s-1}$ all 
vanish. In our case, the ${\rm Gr}^W_{2s}$ are all
constant of rank $1$, and have Hodge structure isomorphic to $\Z(-s)$, see
\cite{guide}.

The full variation of mixed Hodge structure over $U$ then has a composition 
series $\call_0\to\call_2\to\call_4\to\call_6$
with successive quotients of Tate type. For example, $\call_2=W_2\otimes 
\calo_U$ as a mixed Hodge structure fits into the exact sequence 
\begin{equation}
\eqlabel{maps}
\Z(0) \overset{\alpha}\longrightarrow \call_2 \overset{\beta}\longrightarrow \Z(-1)
\end{equation}
that can be described explicitly as follows \cite{delignelimit}. Let $g_0=\alpha(1)$ 
be an integral generator of $W_0\subset W_2$ (\ie, a primitive monodromy invariant 
section of $\calh_\Z$), and $g_1$ be a complementary integral generator of 
$W_2$. Thus, $g_1$ is a multi-valued $\nabla$-horizontal section on which the 
monodromy acts as $N(g_1)= m g_0$ for some non-zero integer $m$. Specifying 
the two-step Hodge filtration on $\call_2$ is equivalent to giving a 
generator $e^1$ of $F^1\subset\call_2 = (\Z g_0 + \Z g_1 )\otimes\calo_U$.
The image of this section under $\beta$ must generate $\Z(-1)$, whose Hodge 
filtration consists only of an $F^1$. So the image cannot vanish, and we can
normalize $e^1$ such that $\beta(e^1)=1$. By the nilpotent orbit theorem mentioned
above, $F^1$ can be continued across $a$. This implies that $e^1$ can be chosen such
that it is single-valued on $U$ and has a limit at $a$. Since $\beta(g_1)=(2\pi\ii)^{-1}$, 
and $\beta(g_0)=0$, this means that we can write
\begin{equation}
\eqlabel{combining}
e^1 = (2\pi\ii) g_1 - m \log q\; g_0
\end{equation}
for some holomorphic function $q$ on $\bar U$ with a simple zero at $a$. A change of 
basis $g_1\to g_1+g_0$ can be compensated by a change of $q$ by an $m$-th root 
of unity, so that the invariant characterizing the extension \eqref{maps} is the 
class
\begin{equation}
\eqlabel{same}
q^m \in {\rm Ext}^1_{\rm VMHS} (\Z(-1),\Z(0)) = \calo_U^*
\end{equation}

The subsequent extensions can be discussed along similar lines \cite{delignelimit},
but we will only present, illustrated with explicit formulas, the final result 
under the two additional assumptions that (i) the monodromy is small, \ie, $m=1$, 
and (ii) the polarization form is unimodular. We can then complete $(g_0,g_1)$ 
to a ``good integral basis'' $(g_s)_{s=0,1,2,3}$ of $\calh_\Z$ 
such that $g_s\in (W_{2s}\cap\calh_\Z)$ (locally around some base point $b\in U$,
or as multi-valued sections) and that is 
primitive in the sense that the matrix $I\in{\rm Mat}_{4\times 4}(\Z(-3))$ 
representing the polarization in this basis has the form
\begin{equation}
\eqlabel{symplectic}
I = \frac{1}{(2\pi\ii)^3}
\begin{pmatrix} 
0 & 0 & 0 & 1 \\
0 & 0 & 1 & 0 \\
0 & -1 & 0 & 0 \\
-1 & 0 & 0 & 0 
\end{pmatrix}
\end{equation}
\ie, $\langle g_0,g_3\rangle=-\langle g_3,g_0\rangle=\langle g_1,g_2\rangle 
=-\langle g_2,g_1\rangle= (2\pi\ii)^{-3}$, while all other pairings
vanish.

The assumption that the monodromy is small implies that we can pick $g_2$ such 
that $M(g_2)= g_2+\kappa g_1$ with $\kappa \in\Z$ (\ie, the coefficient of $g_0$
in $M(g_2)$ can be eliminated by a suitable choice of $g_2$). The condition that
$\langle M (g_3),M(g_i)\rangle=\langle g_3,g_i\rangle$ for $i=0,1,2,3$ is then  
seen to imply that there exists an integer $\lambda$ such that the matrix 
representing $M$ in this basis takes the form
\begin{equation}
\eqlabel{monod}
M = 
\begin{pmatrix}
1 & 0 & 0 & 0 \\
1 & 1 & 0 & 0 \\
0 & \kappa & 1 & 0 \\
\lambda & - \kappa & -1 & 1 
\end{pmatrix}
\end{equation}
Thus the matrix representing $N$ takes the form
\begin{equation}
\eqlabel{logarithm}
N=
\begin{pmatrix}
0 & 0 & 0 & 0 \\
1 & 0 & 0 & 0 \\
-\frac{\kappa}{2} & \kappa & 0 & 0 \\
-\frac{\gamma}{12}& -\frac{\kappa}{2} & -1 & 0
\end{pmatrix}
\end{equation}
where $\gamma = -12 \lambda - 2 \kappa$.

The extension
\begin{equation}
\Z(-2) \to \call_6/\call_2 \to \Z(-3)
\end{equation}
dual to \eqref{maps} w.r.t.\ the polarization, can then be described as follows.
Following common practice, we let $e^3$ be a (single-valued) generator of $F^3$ normalized
such that $\langle g_0, e^3 \rangle = 1$ (the non-vanishing of $\langle g_0,e^3\rangle$
follows from the non-degeneracy of the polarization), and $e^2$ be a complementary
generator of $F^2$ such that $e^2= (2\pi\ii)^{2} g_2 \bmod \call_2$. Under these
conditions, and with the monodromy \eqref{monod} (\ie, $g_3\to g_3-g_2\bmod W_2$), 
we must have
\begin{equation}
\eqlabel{thenn}
e^3 = (2\pi\ii)^3 g_3 + (2\pi\ii)^2 \log q^\vee \; g_2
\bmod\call_2
\end{equation}
where $q^\vee$ is a holomorphic function on $\bar U$ with a simple zero at the puncture. 
Since $F^3$ must be orthogonal to $F^1$ w.r.t.\ $\langle\cdot,\cdot\rangle$, we 
find by pairing $e^3$ with $e^1$ from \eqref{combining} that $q^\vee=q$. 
Then, pairing \eqref{thenn} with $g_1$, we find that 
\begin{equation}
\eqlabel{canonical}
q = \exp  2\pi\ii\frac{\langle g_1, e^3\rangle}{\langle g_0,e^3\rangle}
\end{equation}
This is the standard formula for the so-called canonical coordinate on the neighborhood
$\bar U$ of the maximal degeneracy point $a$. (Given that $q$ has a simple zero
at $a$, it is indeed a good local coordinate to use.)

\def\nablat{{\nabla_{\!\! t}}}
The canonical coordinate is useful to describe the remainder of the mixed
Hodge-Tate structure, as follows. We introduce the logarithmic vector field 
$\delta = \frac{d}{d\log q} = q \frac{d}{dq}$ and denote its contraction with
the Gauss-Manin connection by $\nablat = \nabla (\delta)$. 
Following \eqref{thenn}, we write
\begin{equation}
e^3 = (2\pi\ii)^3 g_3 + (2\pi\ii)^2 \log q\; g_2 -2\pi\ii \cala g_1 - \calb g_0
\end{equation}
for some locally holomorphic functions (periods) $\cala= (2\pi\ii)^2
\langle g_2,e^3\rangle$ and $\calb=(2\pi\ii)^3
\langle g_3,e^3\rangle$. We then define
$e^2=\nablat e^3$ and note that $e^2\in F^2$ by Griffiths transversality. 
Since $\nablat(g_i)=0$,  we find
\begin{equation}
\eqlabel{goback}
e^2 = (2\pi\ii)^2 g_2 -(2\pi\ii) \delta\cala g_1 - \delta\calb g_0
\end{equation}
which shows consistency with our previous definition. Then $F^3\perp F^2$, \ie 
$\langle e^2,e^3\rangle = 0$ implies
\begin{equation}
\cala - \delta\cala\,\log q -\delta \calb = 0
\end{equation}
This relation allows us to express $\cala$ and $\calb$ in terms of the single
function (the prepotential)
\begin{equation}
\eqlabel{single}
\calf = \frac{1}{2} \Bigl(\calb + \log q\; \cala\Bigr)
\end{equation}
Namely
\begin{equation}
\begin{split}
& \cala = \delta \calf \\
\calb =& 2\calf - \log q\; \delta\calf
\end{split}
\end{equation}
A short calculation then shows that
\begin{equation}
\nablat e^2 =  -\delta^3 \calf \bigl((2\pi\ii) g_1 - \log q\; g_0\bigr) 
= -\calc e^1
\end{equation}
where $\calc=\delta^3\calf$, and $e^1$ is from \eqref{combining}. Finally,
\begin{equation}
\nablat e^1 = - g_0 = -e^0
\end{equation}
Thus, the connection matrix in the basis $(e^0,e^1,e^2,e^3)$ takes the form
\begin{equation}
\eqlabel{connection}
\nablat = 
\begin{pmatrix} 0 & 0 & 0 & 0 \\
-1 & 0 & 0 & 0 \\
0 & -\calc & 0 & 0 \\
0 & 0 & 1 & 0
\end{pmatrix}
\end{equation}
The most non-trivial entry of this matrix is the ``normalized Yukawa coupling
in canonical coordinates'', $\calc=\delta^3\calf$.
From \eqref{goback}, we recognize 
\begin{equation}
\eqlabel{recognize}
\exp{\delta \cala} \in {\rm Ext}_{\rm VMHS}(\Z(-2),\Z(-1))
\end{equation}
as the class of the extension $\Z(-1)\to \call_4/\call_0\to\Z(-2)$, and $\calc$
as the logarithmic derivative of this class \cite{delignelimit}. The nilpotent
orbit theorem guarantees that $\exp(\delta\cala)$ is holomorphic and has a zero of
order $\kappa$ at $a$. Alternatively, we can write $\calc$ as the contraction of 
the third iterate of the infinitesimal period mapping
\begin{equation}
\langle \nabla^3 \cdot, \cdot \rangle\in {\rm Sym}^3 T^*\!B \otimes (F^3\otimes F^3)^*
\end{equation}
with $\delta^3\otimes e^3\otimes e^3$, which is perhaps the more frequent 
interpretation \cite{guide}. Namely,
\begin{equation}
\eqlabel{yuk}
\calc = \langle \nablat^3 e^3,e^3\rangle
\end{equation}
All this data can be conveniently summarized in terms of the expansion of the 
prepotential (viewed as a locally holomorphic function on $U$) in the canonical 
coordinate $q$. The monodromy \eqref{monod} dictates that the prepotential be of 
the form
\begin{equation}
\eqlabel{prepotential}
\calf = \frac{\kappa}{6}(\log q)^3 - \frac{\kappa}{4} (2\pi\ii) (\log q)^2
-\frac{\gamma}{24}(2\pi\ii)^2\log q + \bar \varphi 
\end{equation}
where by the nilpotent orbit theorem $\bar \varphi$ is holomorphic on $\bar U$. The 
periods are
\begin{equation}
\eqlabel{periods}
\begin{split}
\langle g_0,e^3\rangle &= 1 \\
(2\pi\ii)\langle g_1,e^3\rangle &= \log q \\
(2\pi\ii)^2\langle g_2,e^3\rangle = \cala = \delta\calf &= \frac{\kappa}{2}(\log q)^2 
-\frac{\kappa}{2}(2\pi\ii)\log q -
\frac{\gamma}{24}(2\pi\ii)^2 + \delta \bar \varphi
\\
(2\pi\ii)^3\langle g_3,e^3\rangle = \calb = 2\calf-\log q\,\delta \calf 
&= -\frac{\kappa}{6}(\log q)^3 - \frac{\gamma}{24}(2\pi\ii)^2\log q 
+2\bar \varphi-\log q\,\delta \bar \varphi
\end{split}
\end{equation}
and the Yukawa coupling
\begin{equation}
\calc = \kappa + \delta^3 \bar \varphi
\end{equation}

\subsection{The limiting mixed Hodge structure}
\label{limMHS}

To describe in greater detail the relation between the canonical coordinate $q$
and a general local coordinate $z$, we return to the continuation of the
Hodge bundle discussed around eq.\ \eqref{untwisted}.

We repeat that the local system $\calh_\Z$ can not be continued to $\bar U$
because of the monodromy $N:(\calh_\Q)_b\to(\calh_\Q)_b$ ($b\in U$). However, 
in conjunction 
with the choice of the local coordinate $z$, the local system can be used to 
induce an integral structure on the fiber $V=\bar\calh_a$ of the continued 
Hodge bundle at $a$.
If $g$ is a local section of $\calh_\Z\subset\calh$ away from the puncture, 
the combination $\bar g = \exp\bigl(-\frac{\log z}{2\pi\ii} N\bigr) g$ is 
horizontal with respect to the untwisted connection $\nabla^c =\nabla
+\frac{N}{2\pi\ii}\frac{dz}{z}$. It thus becomes a section of $\bar\calh$ in a 
neighborhood of $a$, and we put
\begin{equation}
\eqlabel{weput}
\Psi_z(g) = \exp\Bigl(-\frac{\log z}{2\pi\ii}N\Bigr) g \Bigr|_a \in V
\end{equation}
Putting $V_\Z={\rm Im}(\Psi_z)$, the isomorphism $V\cong V_\Z\otimes
\C$ defines an integral structure on $V$. (Following common practice,
we partially suppress the dependence on $z$ in the notation.)

Using $\Psi_z$, we can also carry the monodromy $N$ and associated filtration $W_*$
over to $V_\Q=V_\Z\otimes\Q$. Together with the Hodge filtration $\bar F^*_a$ (which 
we recall does not depend on the choice of $z$), this defines a mixed Hodge structure 
on $V=\bar\calh_a$, known as the limiting mixed Hodge structure (LMHS). It will play 
a central role in the following, so we illustrate it with a few formulae.

Starting from \eqref{periods}, we can express the (multi-valued) basis of integral
sections $(g_s)$ ($s=0,1,2,3$) in terms of the sections $(e^s)_{s=0,1,2,3}$ of 
$\calh$, the prepotential, and $\log q$:
\begin{equation}
\eqlabel{periodmatrix}
\begin{split}
g_0 &= e^0 
\\
(2\pi\ii) g_1 &= e^1 + \log q \, e^0 
\\
(2\pi\ii)^2 g_2 &= e^2 + \delta^2\calf\, e^1 + \delta \calf\, e^0 \\[.1cm]
&= e^2 + \Bigl(\kappa \log q - \frac{\kappa}{2} (2\pi\ii) + \delta^2\bar \varphi\Bigr) e^1 
\\
&\qquad\qquad\qquad\qquad
+ \Bigl(\frac{\kappa}{2}(\log q)^2 -\frac{\kappa}{2} (2\pi\ii)\log q -\frac{\gamma}{24}
(2\pi\ii)^2 + \delta \bar \varphi\Bigr) e^0
\\
(2\pi\ii)^3 g_3 &= e^3 - \log q \, e^2 + (\delta\calf - \log q\,\delta^2\calf)
e^1 + (2\calf - \log q\,\delta \calf) e^0 \\[.1cm]
&= e^3 - \log q\, e^2 + 
\Bigl( -\frac{\kappa}{2} (\log q)^2 -\frac{\gamma}{24}(2\pi\ii)^2 +
\delta\bar \varphi- \log q \, \delta^2 \bar \varphi \Bigr) e^1 
\\
&\qquad\qquad\qquad\qquad
+ \Bigl(-\frac{\kappa}{6}(\log q)^3 -\frac{\gamma}{24} (2\pi\ii)^2 \log q + 2\bar \varphi
-\log q\,\delta \bar \varphi\Bigr) e^0
\end{split}
\end{equation}
(Using \eqref{connection}, it is easy to check that the $g_s$ are horizontal.)
Then, in correspondence with our (multi-valued) basis $(g_s)$, we introduce
\begin{equation}
\bar g_s = \exp \bigl(-\frac{\log z}{2\pi\ii}  N\bigr) g_s
\end{equation}
which by construction can be continued as sections of $\bar\calh\to\bar U$.
This untwisting amounts simply to the replacement of $\log q$ with $\log q/z$
in \eqref{periodmatrix},\footnote{In contrast, the formation of the nilpotent orbit
(see footnote \ref{nilpotent} on page \pageref{nilpotent}) amounts to keeping 
{\it only} the logarithmic terms.} so that we have
\begin{equation}
\eqlabel{ala}
\begin{split}
\bar g_0 &= e^0 
\\
(2\pi\ii) \bar g_1 &= e^1 + \log \frac{q}{z} \, e^0 
\\
(2\pi\ii)^2 \bar g_2 
&= e^2 + \Bigl(\kappa \log \frac {q}{z} - \frac{\kappa}{2} (2\pi\ii) + 
\delta^2\bar \varphi\Bigr) e^1 
\\
&\qquad\qquad\qquad
+ \Bigl(\frac{\kappa}{2}(\log \frac{q}{z})^2 -\frac{\kappa}{2} 
(2\pi\ii)\log \frac{q}{z} -\frac{\gamma}{24}
(2\pi\ii)^2 + \delta \bar \varphi\Bigr) e^0
\\[.1cm]
(2\pi\ii)^3 \bar g_3 
&= e^3 - \log \frac{q}{z}\, e^2 + 
\Bigl( -\frac{\kappa}{2} (\log \frac{q}{z})^2 -\frac{\gamma}{24}(2\pi\ii)^2 +
\delta\bar \varphi- \log \frac{q}{z} \, \delta^2 \bar \varphi \Bigr) e^1 
\\
&\qquad\qquad\qquad
+ \Bigl(-\frac{\kappa}{6}(\log \frac{q}{z})^3 -\frac{\gamma}{24} (2\pi\ii)^2 
\log\frac{q}{z} + 2\bar \varphi
-\log \frac{q}{z}\,\delta \bar \varphi\Bigr) e^0
\end{split}
\end{equation}
And indeed, since $q$ and $z$ both vanish to first order at $a$, $\lim q/z=c\neq 0$
exists, which leads to the limiting period matrix
\begin{equation}
\eqlabel{limitingperiod}
\Pi =(\Pi_{st}) := \begin{pmatrix}
1 & 0 & 0 & 0 \\
\frac{\log c}{2\pi\ii} & \frac{1}{2\pi\ii} & 0 & 0 \\
\frac{\kappa}{2} \frac{(\log c)^2}{(2\pi\ii)^2}-\frac{\kappa}{2}\frac{\log c}{2\pi\ii}
-\frac{\gamma}{24}  
& \kappa\,\frac{\log c}{(2\pi\ii)^2} - \frac{\kappa}{2}\frac{1}{2\pi\ii}
& \frac{1}{(2\pi\ii)^2} & 0 \\
-\frac{\kappa}{6} \frac{(\log c)^3}{(2\pi\ii)^3} -\frac{\gamma}{24}\frac{\log c}{2\pi\ii}
+ \frac{2\zeta}{(2\pi\ii)^3}
& \frac{\kappa}{2} \frac{(\log c)^2}{(2\pi\ii)^3} -\frac{\gamma}{24}\frac{1}{2\pi\ii}
& - \frac{\log c}{(2\pi\ii)^3} 
& \frac{1}{(2\pi\ii)^3}
\end{pmatrix}
\end{equation}
with $\lim_{z\to 0}(\bar g_s-\Pi_{st}e^t)=0$.
Here $\zeta = \lim_{z\to 0} \bar \varphi$ is a complex number and the only entry that is not 
determined by considerations of local monodromy.

\subsection{Algebraic cycles and extensions}
\label{algcycext}

The starting point of this section was a smooth family $\pi:Y\to B$  of Calabi-Yau 
threefolds over a quasi-projective complex curve $B$, admitting a semi-stable 
compactification $\bar\pi:\bar Y\to\bar B$, with a distinguished boundary point
$a\in\bar B\setminus B$ of maximal degeneration. A natural extension of this
situation, considered in \cite{arithmetic}, is by a complex
algebraic surface $i:C\to Y$, with the following properties:
\\
(i) The composition $\pi\circ i : C \to B$ is a semi-stable flat
family of curves, and the situation admits a semi-stable compactification over
$\bar B$.
\\
(ii) On a dense open subset, $i:\overset{\,\circ}{C}\to Y$ is a smooth immersion,
and $\pi\circ i:\overset{\,\circ}{C}\to\overset{\,\circ}{B}$ is a smooth family. In other 
words, in the generic member $Y_b:=\pi^{-1}(b)$ of the family, $i_b:C_b=(\pi\circ i)^{-1}
(b)\to Y_b$ is an immersed curve. It is important that we allow the fibers $C_b$ to 
be reducible. We assume that the irreducible components of $C_b$ are
homologically equivalent to each other in $Y_b$. In other words, writing $C_b= 
\cup_{k} C_{b,k}$, we 
assume that $[i_b(C_{b,k})]-[i_b(C_{b,k'})]=0\in H_2(Y_b,\Z)$.
\\
(iii) There exists an embedded surface $\bar C_0\hookrightarrow \bar Y$ such that the
composition $\bar C_0\to \bar B$ is a smooth family, with irreducible fibers. We denote
these fibers by $C_{b,0}$ and assume that for generic $b$, some fixed positive 
multiple of $C_{b,0}$ is homologically equivalent to the components of $C_b$ in 
$Y_b$. 

In the following, we'll pretend that this multiple is $1$. Moreover, we shall 
restrict to a simply connected neighborhood $\bar U$ of $a$ such that $U=\bar U
\setminus \{a\}\subset \overset{\,\circ}{B}$ and such that for $b\in U$, the 
components $C_{b,k}$ are all smooth. We allow ourselves to drop a subset of the
components of $C\to U$ (wihtout new notation), in order to satisfy 
conditions further specified below. We also assume that $C_{a,0}$ is smooth.

To this configuration $(Y,C)\to U$ is now attached a variation of mixed Hodge 
structure $\hat\calh$, which can be thought of as an extension of a pure Hodge 
structure $\cali$ of weight $4$ by the pure Hodge structure $\calh$ of weight 
$3$ attached to $Y\to U$. It looks as follows.
\\
(1) The local system $\cali_\Z=(\pi\circ i)_*\Z$ is free of rank equal to the number 
of components of $C_b$. Its fibers are  $H^0(C_b,\Z)\cong H_2(C_b,\Z)$ by Poincar\'e
duality on $C_b$. The extension of local systems
\begin{equation}
\calh_\Z\to \hat\calh_\Z \to \cali_\Z
\end{equation}
can be identified at each fiber as the exact sequence in relative homology
\begin{equation}
\eqlabel{relativehom}
0\to H_3(Y_b,\Z) \to \check{H}_3(Y_b,C_b,\Z) \to H_2(C_b,\Z) \to 0
\end{equation}
Here, we have ``based'' the relative homology group by letting
\begin{equation}
\eqlabel{based}
\check{H}_3(Y_b,C_b,\Z) := H_3(Y_b,C_b\cup C_{b,0},\Z)
\end{equation}
and we have identified $H_2(C_b,\Z)\cong {\rm Ker}(H_2(C_b\cup C_{b,0},\Z)\to
H_2(Y_b,\Z))$ in an obvious way (\ie, by using that $C_{b,0}$ is homologous to all the
irreducible components of $C_b$). The identification of the extension of local
systems with \eqref{relativehom} is induced by Poincar\'e duality from the 
exact sequence in cohomology
\begin{equation}
\eqlabel{poincaredual}
0 \to H^3(Y_b,\Z) \to \check{H}^3(Y_b\setminus C_b,\Z) \to H^0(C_b,\Z) \to 0
\end{equation}
(2) Although born as an $H^0$, $\cali=\cali_\Z\otimes\calo_U$ in fact has weight $4$ 
as a consequence of the embedding in $Y_b$, and is purely of Hodge type $(2,2)$. 
(By Poincar\'e duality on $Y_b$, the components of $C_b$ are represented by 
4-cochains or ``currents'' with delta-function support on the $C_{b,k}$). As a 
consequence, the Hodge filtration $\hat F^*$ on $\hat\calh=\hat\calh_\Z\otimes 
\calo_U$ satisfies $\hat F^s/\hat F^{s+1}= F^s/F^{s+1}$ except for $s=2$, and 
we have $\hat F^2/ F^2=\cali$.
\\
(3) The polarization does not extend in a canonical way to all of $\hat\calh$.
However, given that $F^2\perp F^2$, and the above properties of the Hodge
filtration, it makes sense to extend the pairing with $F^2$ as a bilinear
form from $\calh\times F^2$ to $\hat\calh\times F^2$. (Namely, we define
the pairing to be $0$ on $(\hat\calh/\calh)\times F^2=(\hat F^2/F^2)\times F^2$.) 
Doing so\footnote{Geometrically, this is accomplished by integrating three-forms 
representing elements of $F^2$ against three-chains of boundary $C_{b,k}-
C_{b,0}$, see \cite{normal,arithmetic}.} allows us to
identify extensions of Hodge structure in ${\rm Ext}^1_{\rm VMHS}(\Z(-2),\calh)$ 
with $J\cong(F^2)^*/\calh_\Z$ (Abel-Jacobi map).

Needless to say, the Gauss-Manin connection extends to $\hat\calh$ with 
horizontal sections $\hat\calh_\C$, and Griffiths transversality $\nabla 
\hat F^s\subset \hat F^{s-1}\otimes\Omega_U$.

As before, the monodromy theorem guarantees that the extended monodromy operator 
$\hat M:\hat\calh_\Z\to \hat\calh_\Z$ is quasi-unipotent. It preserves 
$\calh_\Z$ and agrees with $M$ there. On the quotient, $\cali_\Z$, the monodromy
is of finite order\footnote{This follows from the nilpotent orbit theorem because
$\cali$ is Hodge-Tate.}, which we denote by $r$. In other words, $r$ is the smallest
positive integer such that $\hat M^r$ is unipotent, and we define $\hat N=\log 
\hat M^r$. Note that $\hat N|_{\calh_\Q}= r N$. We assume (possibly after dropping
some of the fibers) that all the orbits of $\hat M$ on $\cali_\Z$ are of the same 
order, and denote the number of orbits by $\hat d$.

We may trivialize this finite monodromy of $\cali_\Z$ by passing to an $r$-fold cover 
$(\hat U \to U )\subset (\bar {\hat U} \to \bar U)$ branched at $a$. We won't 
introduce new notation for those parts of the data that pull back trivially 
to $\hat U$, but for the local system $\hat \cali_\Z$ whose rank drops from $r\hat d$
to $\hat d$. 
The extension of variation of Hodge structures is now of the form
\begin{equation}
\eqlabel{hatext}
\calh \overset{{\rm id}}{\longrightarrow} \hat\calh 
\overset{\beta}{\longrightarrow} \Z(-2)^{\hat d}
\end{equation}
To describe it explicitly, we first extend our basis $(g_s)_{s=0,\ldots, 3}$ with 
$g_s\in W_{2s}\subset \hat W_{2s}$ and monodromy \eqref{monod} by a collection 
of complementary generators $(h_k)_{k=1,\ldots,\hat d}$. We pause to explain certain
(``torsion'') subtleties that arise in the choice of the $h_k$. 

Because $\hat N$ projects to $0$ on $\hat\cali_\Q$, the extension of the monodromy 
filtration is ``concentrated in the middle''. Namely, $\hat W_*$ satisfies $\hat 
W_k=W_k$ for $k<3$, $\hat W_k/\hat W_3=W_k/W_3$ for $k>3$, and $\hat W_3/\hat 
W_2\cong \hat\cali_\Q$. However, we can not necessarily assume that the extending
generators $h_k$ are both integral generators of $\hat\calh_\Z/\calh_\Z$ {\it and} 
contained in $\hat W_3\subset\hat\calh_\Q$: The image of integral extending 
generators under monodromy might be contained in $W_2$, and obtaining generators 
of $\hat W_3$ might require a change of basis that is rational but not in general integral. 
This point was emphasized in \cite{ggk2}, and explicit examples illustrating the 
phenomenon can be found in \cite{jefferson}. For simplicity, we will here assume that 
the $h_k$ are both in $\hat W_3$ and that their images under $\beta$ generate 
$(2\pi\ii)^{-2}\Z^{\hat d}\subset\hat\cali=\hat\cali_\Z\otimes\calo_U$. 

This assumption does however not remove the subtleties completely. Monodromy acts by 
$\hat N(h_k) = a_k g_0 \in \hat W_1=W_0$, with $a_k\in\Z$. The $h_k$ are canonical 
up to the addition of integral multiples of 
$g_1$ and $g_0$, which changes the integers $a_k$ by integral multiples of $r\cdot m$. 
(Recall that $N(g_1)=m g_0$, and so $\hat N(g_1)=r m\, g_0$.) Thus, even assuming 
that the monodromy is small ($m=1$), we cannot take $a_k=0$ in general. We do not
wish to assume $r=1$ because it would exclude most of the examples of \cite{arithmetic}.
(Although the proofs are not significantly more complicated without the assumption.)

In opposition, we extend our basis $(e^s)_{s=0,\ldots, 3}$ with $e^s\in 
F^s\subset \hat F^s$ by a collection of complementary generators $(f_k)_{k=1,\ldots,
\hat d}$ of $\hat F^2$ that agree with $((2\pi\ii)^2h_k)$ $\bmod \calh$, in other words,
that $(2\pi\ii)^2 \beta(h_k) = \beta(f_k)$. The Hodge structure is specified by
lifting this relation between the $h_k$ and $f_k$ to $\hat\calh$. In this process,
the $F^2$-part of $f_k$ remains arbitrary, so that we can choose the $f_k$ such that
(\cf, \eqref{periodmatrix}),
\begin{equation}
\eqlabel{specify}
(2\pi\ii)^2 h_k = f_k + \calv_k e^1 + \calw_k e^0 
\end{equation}
where the $\calw_k$ and $\calv_k$ are locally holomorphic functions on $U$, and
the nilpotent orbit theorem guarantees that the $f_k$ continue across $a$.

To analyze the behaviour of \eqref{specify} at the point $a$ of maximal degeneration
more precisely, we use the canonical coordinate $q$ on $\bar U$ (see eq.\ 
\eqref{canonical}), or rather, its lift $q={\hat q}^r$ to $\bar{\hat U}$. Since
$\nablat h_k=0=\nablat e^0$, $\nablat e^1= - e^0$, and $\nablat f_k\in \hat F^1$, 
we must have
\begin{equation}
\calv_k = \delta\calw_k
\end{equation}
and
\begin{equation}
\eqlabel{ontheotherhand}
\nablat f_k = - \delta\calv_k\, e^1 = -\delta^2\calw_k\, e^1
\end{equation}
where $\delta = q\frac{d}{dq}$ as before. The monodromy $\hat N(h_k) = a_k g_0=a_k e^0$
implies that $\calw_k$ must be of the form (\cf, \eqref{prepotential})
\begin{equation}
\eqlabel{mustbe}
\calw_k = a_k (2\pi\ii) \log q^{1/r} + \bar w_k
\end{equation}
where $\bar w_k$ is single valued on $\hat U$, and the nilpotent orbit theorem 
implies that it continues holomorphically to $\bar {\hat U}$.

In eq.\ \eqref{specify}, the combination 
\begin{equation}
\eqlabel{cancom}
\hat\nu_k = \calv_k e^1 + \calw_k e^0 \in \calh \cap W_2
\end{equation}
can be viewed as the normal function that generally classifies extensions
of Hodge structures by algebraic cycles of this type, see \cite{normal}. More 
precisely, the normal function is the image $\nu_k$ of $\hat \nu_k$ in the 
intermediate
Jacobian $J=\calh_\Z\backslash\calh\slash F^2$. As explained in \cite{ggk2},
the maximal degeneration of $\calh$ at $a$ makes the lift \eqref{cancom}
well-defined modulo $\calh_\Z\cap W_2$ instead of $\calh_\Z$.
We also note that in terms of the pairing on $\hat\calh\times F^2$ 
discussed above, we have
\begin{equation}
\eqlabel{superpotential}
\calw_k = (2\pi\ii)^2 \langle h_k, e^3\rangle 
\end{equation}
which is the definition of the ``superpotential'' (the truncated normal function)
used in \cite{normal}.

Lastly, the (Griffiths) infinitesimal invariant, which is the analogue of
the Yukawa coupling \eqref{yuk} characterizing the variation of Hodge structure 
locally is the combination\footnote{Formally, the Griffiths infinitesimal 
invariant is the
class of $\nabla\hat\nu\in F^1$ in ${\rm Ker}(\nabla\wedge)/({\rm Im}(\nabla)$,
where the equivalence accounts for the a priori ambiguity of $f_k$ in 
$F^2$. Maximal degeneration provides a canonical lift, and the expression
\eqref{griff} is this ``normalized canonical representative of the Griffiths 
infinitesimal invariant in canonical coordinates''.}
\begin{equation}
\eqlabel{griff}
\cald_k = \langle \nablat^2 \hat \nu_k, e^3\rangle = \delta^2 \calw_k
\end{equation}
Finally, if we choose as in subsection \ref{limMHS} a general local coordinate $z$ 
vanishing to first order at $a$, we can define the limiting Hodge structure by 
untwisting the local system \` a la \eqref{weput}, \eqref{ala}, 
\begin{equation}
(2\pi\ii)^2 \bar h_k = f_k + \Bigl(\frac{a_k}{r} (2\pi\ii) + \delta\bar w_k\Bigr)
e^1 + \Bigl(a_k (2\pi\ii) \log \frac{q^{1/r}}{z^{1/r}} + \bar w_k\Bigr) e^0
\end{equation}

\subsection{Integrality statements}
\label{intdef}

We started the discussion with a family of Calabi-Yau varieties $Y\to B$
over a complex curve $B$, admitting a semi-stable compactification
$\bar Y\to\bar B$. Assuming that the boundary point $a\in\bar B\setminus B$
is a point of maximal degeneration, we reviewed how the variation of Hodge 
structure is encoded locally in a set of holomorphic functions of a local
coordinate $z$. Among these functions are the canonical coordinate $q$ 
and the normalized Yukawa coupling $\calc$. We then extended this family 
of varieties by a family of algebraic cycles $C \subset Y$ varying continuously 
with $Y$ over $B$, and reviewed how the associated extension of Hodge structure 
can be encoded in another holomorphic function, the infinitesimal invariant 
$\cald_k$. (Here, $k$ is an index labelling components of the generic fiber
of $C$, see previous subsection for precise definitions.)

We now add the assumption that the maximal degeneration is defined over the
integers, and that $z$ is an integral coordinate on $B$ (we give the precise 
definitions momentarily). This implies that the functions $q$ and $\calc$,
when expanded around $z=0$, are power series with rational coefficients.
It was proven in \cite{ksv,vadim} that for all primes $p>3$ for which 
the reduction $\bmod p$ is smooth, the canonical coordinate $q(z)$ has $p$-integral 
coefficients, and that when the normalized Yukawa coupling is re-written as a
power series in $q$, the coefficients satisfy congruence relations equivalent 
to the statement that the non-constant part of $\calc$ is the third logarithmic 
derivative of a $3$-function at $p$ in the sense of section \ref{framing}.

In \cite{sv2}, it was shown that under the assumption that the degeneration of
the cycle $C$ is also defined over the integers, the infinitesimal invariant
is the second derivative of a $2$-function at $p$.
The main result of the remainder of this paper is the generalization to the 
situation in which the cycle is not defined over $\Q$.

\medskip

To give precise definitions and state the results, we liberate the notation 
from the previous subsection, and rephrase the assumptions in scheme-theoretic 
language. The semi-stable map of complex algebraic varieties $\bar\pi:\bar 
Y\to\bar B$ can be viewed as a semi-stable morphism of schemes over $\spec\C$. 
We assume that the field of definition of this morphism is $\Q$, which means
that there exists a semi-stable morphism $\bar\pi_\Q:\bar Y_\Q\to\bar B_\Q$
over $\spec\Q$ together with an isomorphism
$\bar\pi_\Q\times_{\spec\Q} \spec\C\cong \bar\pi$. We mention that while it
is in general not easy to identify the (smallest) field of definition of
any given scheme, from the point of view of $\Q$, the important property of 
$\bar Y$ is that $\bar Y_\Q\times_{\spec\Q}\spec\bar\Q$ remains irreducible as a 
scheme over $\Q$, where $\bar\Q$ is the algebraic closure of $\Q$.
We also assume that
the boundary point $a$ is rational, which means that it is the complexification
of a section $a_\Q:\spec\Q\to\bar B_\Q$. 

An important consequence of these assumptions is that the singular fiber of
the family, $\bar Y_a= \bar Y \times_{\bar B} \spec\C$ (where $\spec\C\overset{a}{\to}
\bar B$) is also
defined over $\Q$. Furthermore, letting $z$ be a rational local coordinate on $\bar B$
vanishing at $a$ (namely, given an identification of a neighborhood
of $a_\Q$ in $\bar B_\Q$ with $\spec \Q[[z]]$), the localization of $\bar Y$ at $a$,
$\bar Y \times_{\bar B} \spec\C[[z]]$ is also defined over $\Q$.\footnote{It
is an interesting question whether all maximal degenerations of
complex algebraic families of Calabi-Yau 3-folds are necessarily defined over
$\Q$. Mirror symmetry seems to strongly suggest that this is true, but we cannot
imagine any reason for this statement from the point of view of the B-model.}

We can not, in general, maintain these assumptions after extension by the algebraic
cycle $(i:C\to Y)\subset(\bar C\to\bar Y)$. Following the assumptions of subsection 
\ref{algcycext},
and with similar notational conventions, we denote the localization of $(\bar Y,
\bar C)$ to the formal neighborhood $\bar D = \spec\C[[z]]$ of $a$ by the same
letters. This formal neighborhood (or its underlying rational analogue)
takes the place of the complex neighborhood 
$\bar U$ from subsection \ref{algcycext}.
We allow the finite part of the Stein factorization of the map $\bar C 
\to \bar D$ to be branched at $a$ with ramification index $r$, and denote the
$r$-fold cover $z=\hat z^{r}$ of $\bar D$ by $\bar{\hat D}=\spec \C[[\hat z]]$.
This corresponds to the monodromy of order $r$ on $\cali_\Z$ from
subsection \ref{algcycext}. In that subsection, we had allowed the fibers of ${C} 
\to {\hat U}$ to have $\hat d\ge 1$ irreducible components. We now specify 
that $\bar {C}\to \bar {\hat D}$ should be the complexification of a scheme
that is irreducible over $\spec\Q$, but whose field of definition $K$ (the smallest 
intermediate field such that further extension leaves components of the fiber 
irreducible) can be a finite extension of $\Q$.
We emphasize that $K$ need not be Galois over $\Q$ and that 
its degree, $d=[K:\Q]$ need not equal $\hat d$. Rather, the complex cycles of 
subsection \ref{algcycext} each corresponds to a different embedding 
$K\hookrightarrow \C$, which could be permuted by the monodromy of order $r$ 
around $a$ (see \cite{arithmetic,jefferson} for examples of this phenomenon). 

In the algebraic setup, the irreducibility of the cycle implies that,
after localization to $a$, the extension classes $\calv_k$ from \eqref{specify}
for different values 
of $k$ fit together to a single formal power series on the ``extended disk'' 
$\bar{\hat D}^K=\spec K[[\hat z]]$. The same is true for their derivatives, 
$\cald_k$, but {\it not} the superpotential $\calw_k$ itself, which generically 
includes (apart from the $\log$-term) 
a (conjecturally) transcendental constant, see \cite{ggk2,jefferson}. 
We shall denote these functions by $\calv$, $\cald$.
We continue to denote by $q$, $\calc$ the localization of the Hodge theoretic
extension classes from subsection \ref{VHS} to the formal neighborhood 
of $a$, possibly pulled back to $\bar{\hat D}$.  We have

\begin{lemma}
\label{rationalpowers}
$q\in z\Q[[z]]\subset z\C[[z]]$, $\calc\in \Q[[q]]\subset\C[[z]]$, and
$\cald\in \hat q K[[\hat q]]\subset\!\!\!\subset\!\!\!\subset \hat z\C[[\hat z]]$.
\end{lemma}

\begin{remark}
The first of these statements depends crucially on the assumption of smallness of monodromy 
over $\Z$, \ie, $m=1$ in \eqref{same}. (In general, we can only prove
that $q^m\in z\Q[[z]]$, see \cite{vadim}.) An equivalent statement is
$q'(0)\in\Q$, which we will assume in the following.

It might happen that $\cald$ has coefficients in a subfield of $K$, for instance 
if the algebraic cycle is rationally equivalent to a cycle defined over $\Q$.
\end{remark}

In order to formulate the main integrality statements, we have to continue our
families from the generic point $\spec\Q\hookrightarrow\spec\Z$ to some larger 
set of ``good'' primes. We exclude any primes at which $Y_\Q\to D_\Q$ or $C_\Q\to 
D_\Q$ are 
not smooth, or their compactifications are not semi-stable, all divisors of $r$, 
as well as those of the discriminant of the extension $K/\Q$. In order to apply
the $p$-adic Hodge theory (section \ref{pBmodel}), we also need to exclude all 
prime $p\le\dim (Y/B) + 1$ (in our case, this excludes $2$ and $3$). We now let 
$(N)\subset\spec\Z$ be the union of all these excluded points and assume that
there exist schemes over $S:=\spec\Z[N^{-1}]=\spec Z\setminus(N)$ whose base 
change to $\C$ gives rise to the complex varieties from above.
We also assume that the coordinate $z$ 
has been chosen such that it is integral everywhere on $S$.

The following result was proven in \cite{ksv,vadim}.\footnote{A comment for
readers on return from one of the later sections: The theorems state that 
the power series $q$, $\calc$, $\cald$ as defined via {\it complex} Hodge
theory have the indicated integrality property, though the proofs depend on
$p$-adic methods. The outcome of section \ref{pBmodel} is that the
complex power series essentially agree with the $p$-adic ones. Section \ref{proof}
establishes integrality of the $p$-adic series. We felt that carrying the weight of 
subscripts $\C$, $\Q$, $\Z_p$ offered additional clarity only rather temporarily.}
\begin{theorem}
\label{previous}
For all primes $(p)\in S$ (\ie, those with $(p,N)=1$), we have
$q\in z\Z_p[[z]]$, and there exists a formal power series $\psi_p(q)\in
q\Z_p[[q]]$ such that
\begin{equation}
\calc(q^p) -\calc(q) = \delta^3 \psi_p
\end{equation}
In the terminology of section \ref{framing}, the non-singular part of the 
prepotential \eqref{prepotential} (without the constant term) is a 3-function 
at $p$.
\end{theorem}

In this paper, we extend the generalization of the integrality result \cite{sv2}
to the situation with a non-trivial residue field.
\begin{theorem}
\label{origin}
For all primes $(p,N)=1$, there exists a formal power series $\omega_p(q)\in
\hat q\G_p[[\hat q]]$ such that, in the notation of section \ref{framing}, we have
\begin{equation}
\frob_p \cald - \cald = \delta^2\omega_p
\end{equation}
In other words, the non-singular part of the superpotential (without the
constant term) is a $2$-function at $p$ with coefficients in $K\subset K_p$.
\end{theorem}

\begin{remark}
\label{mullerstach}
Theorem \ref{origin} provides an answer of sorts to the question of
existence of analytic $2$-functions that we have posed in subsection \ref{basis}.
To explain this, we recall that the Picard-Fuchs operator\footnote{As before,
we only make statements for one-dimensional moduli spaces, but they all
admit fairly obvious generalizations.} annihilating 
the periods, which in the canonical coordinate takes the form
\begin{equation}
\calp :=\delta^2 \calc^{-1} \delta^2
\end{equation}
is a differential operator with algebraic coefficients when written in 
a global algebraic coordinate $z$ over $B$. Application to the superpotential
\eqref{superpotential}, \eqref{griff} does not return $0$ in general.
Let us define instead for each $k$,
\begin{equation}
j_k := \frac{d}{dz} \bigl(\delta \calc^{-1}\delta^2 \calw_k \bigr) 
\end{equation}
Then the $j_k$ are local power series that determine $\calw_k$ up to periods,
\ie, up to the constant and logarithmic terms in \eqref{mustbe}.
On general grounds, the $j_k$ are the local expansions of algebraic
functions over the moduli space, explicitly calculated in 
\cite{normal,arithmetic}. In other words, the $\calw_k$ are at the same
time (explicitly) $2$-functions in the variable $q$, and (explicitly) 
analytic in $z$. We believe that this deserves further attention.
\end{remark}

\section{The (extended) \texorpdfstring{$p$}{p}-adic B-model}
\label{pBmodel}

The purpose of this section is to review the strategy of the integrality 
proofs of \cite{ksv,vadim,sv2}, 
and to explain some background on $p$-adic Hodge theory and its comparison
with the more familiar complex Hodge theory. 

\medskip

Theorems \ref{previous} and \ref{origin} make integrality statements about
formal power series that are attached to families of complex algebraic
varieties $(Y,C)\to B$ by a Hodge theoretic construction described in 
subsections \ref{VHS} and \ref{algcycext}. 
The relevant assumption is that the complex algebraic 
varieties in fact come from a more abstract scheme that provides an 
underlying algebraic integral structure, locally around the point of 
maximal degeneration.

It is important to note that these power series cannot entirely be attached to 
the algebraic structure alone. One of the two main ingredients of the variation
of Hodge structure is the topological integral structure (the 
local systems $\calh_\Z\subset\hat\calh_\Z$), and to define this topological integral 
structure, we require the fine topology of the complex numbers. 
In the way we have explained, the power series arise as ``periods'' during 
the pairing between the algebraic and topological cohomology groups
(see eq.\ \ref{periods}). Adding the algebraic cycle leads to an extension 
of the ``local period ring'' from $\Q[[z]]$ to $K[[\hat z]]$.

\medskip

The clue for an explanation of the integrality is included in the formulation
of the theorems, via their reference to a (``good'') prime number ($p$) and an
identification of the periods as power series with $p$-adic coefficients. The
main idea of \cite{ksv} is to relate these power series with the ``periods''
in the $p$-adic world, \ie, with the coefficients involved
in the comparison between the algebraic (deRham) cohomology and topological 
(\'etale or crystalline) cohomology. The $p$-adic integrality of the power 
series then follows from the properties of these $p$-adic 
theories. We will present the relevant calculations in the next section, and here
attempt to convey an idea of the underlying concepts.

The crux is that a priori, the complex and $p$-adic definitions of the periods have 
little to do with each other: As we just mentioned, the topological integral structure
in the complex setting comes from viewing the algebraic varieties as complex (in fact, 
real) topological manifolds. In contrast, in the $p$-adic setting, the role of the 
topological integral structure is played by the Frobenius symmetry acting on the 
algebraich cohomology groups. (We will explain this in more detail below.) 
Thus, a major step of the integrality proof is to show that the functions defined 
in the complex and $p$-adic setting in fact agree.

The identification between the complex and $p$-adically defined functions in turn 
divides in two parts. One first verifies that the functions satisfy the same
differential relations in the neighborhood of the point of maximal degeneration.
This is essentially a consequence of the fact that the differential equation
satisfied by the periods (the homogeneous and inhomogeneous Picard-Fuchs
equations) have rational and algebraic coefficients respectively. Then, one
remains with checking that the initial conditions at the point of maximal
degeneration also agree. This part (which is technically the hardest and will
not be reviewed here) involves a comparison between the complex and $p$-adic 
Hodge structures in the strict degeneration limit. 

\subsection{Logarithmic de Rham cohomology over \texorpdfstring{$\Q$}{Q}}
\label{overQ}

To begin with, we isolate those parts of the complex Hodge theory that can
be defined purely algebraically, and which we can then complete to the
$p$-adic setting (instead of to $\C$). For our purposes, it will be 
sufficient to work over the formal disk $\bar D=\spec\Q[[z]]$, thought of as
a rational 
neighborhood of $a=(z)$ in $\bar B$ as explained above. Thus, $\bar\pi: 
\bar Y \to \bar D$ is a semi-stable morphism such that $\pi: Y\to D$ 
(with $D=\spec\Q[[z,z^{-1}]]=\bar D\setminus a$) is a smooth family of 
Calabi-Yau schemes of relative dimension $3$, $a\cong\spec\Q\hookrightarrow
\bar D$ is the closed point and $\bar Y_a=\bar Y\times_{\bar D}\spec \Q$ 
is the singular fiber.

In this setting, the rational analogue of the continued Hodge bundle 
$\bar\calh\to\bar D$ can be defined, without reference to topology, via logarithmic 
de Rham cohomology\footnote{In potential conflict with previous or later
notation, all schemes are taken over $\Q$ in this subsection, unless stated 
or implied otherwise by context.}
\begin{equation}
\label{vialog}
\bar\calh = H^3_{\rm log}(\bar Y/\bar D) = R^3\bar\pi_*\bigl(
(\Omega^*_{\bar Y/\bar D}(\log(\bar Y_a)),d)\bigr)
\end{equation}
$\bar\calh$ is a vector bundle over $\bar D$ and comes equipped with (see 
\cite{faltings}, and \cite{vadim} for more complete information)
\\
$\ast$ a decreasing filtration $\bar\calh=\bar F^0\supset \cdots\supset \bar F^3$
by subbundles $\bar F^s\to\bar D$ with ${\rm rank}\bar F^3=1$
\\
$\ast$ a flat logarithmic connection $\nabla:\bar\calh\to\bar\calh\otimes
\Omega^1_{\bar D}(\log a)$ satisfying Griffiths transversality
$\nabla \bar F^s \subset\bar F^{s-1}\otimes \Omega^1_{\bar D}(\log a)$
\\
$\ast$ a perfect pairing $\langle\cdot,\cdot\rangle:\bar\calh\times\bar\calh
\to H^6_{\rm log}(\bar Y/\bar D) \cong\calo_{\bar D}$ 
with $\langle \bar F^s,\bar  F^{4-s} \rangle =0$

We denote the fiber of $\bar\calh$ at $a$ by $V:=\bar\calh_a$ which at this point
is a $\Q$-vector space of dimension $4$. By the assumption of semi-stability, the 
residue $N_{\rm dR}={\rm Res}_a(\nabla):V\to V$ 
of the connection is nilpotent. Since, after complexification, $N_{\rm dR}$
is related to the logarithm of the monodromy of the local system from 
section \ref{VHS} by $N_{\rm dR}= -\frac{1}{2\pi\ii} N$ (see eq.\ \eqref{untwisted}), 
and we have assumed
that the complex degeneration has maximal unipotent monodromy, it follows
that $N_{\rm dR}$ has maximal rank $3$, and, just as in the complex case, 
induces a weight filtration $W_*$ on $V$.

This weight filtration can be used to reconstruct a basis of (``algebraically
rational'') sections $(e^s)$ of $\bar\calh$ over $\bar D$: We begin by letting 
$e^0$ be a parallel section of $\bar\calh=\bar F^0$ whose restriction to $a$ 
generates the one-dimensional subspace $W_0:={\rm Ker}(N_{\rm dR})={\rm Im}(N_{\rm dR}^3)$ 
of $V$. Note that $e^0$ is unique up to a rational number and generates a one-dimensional
subbundle of $\bar\calh$. We then let $e^1$ be a section of $\bar F^1$ such that 
\begin{equation}
\label{indeed}
N_{\rm dR}(e^1(a)) = -e^0(a)
\end{equation}
and $\nabla e^1 \in e^0 \otimes \Omega^1(\log a)$. In other words, the image of 
$e^1$ in the quotient $\bar\calh/e^0$ is parallel w.r.t.\ the induced connection. 
Writing
\begin{equation}
\label{induced}
\nabla e^1 = - e^0 \otimes d\log q_\Q
\end{equation}
determines a ``rational flat coordinate'' $q_\Q\in\calo_{\bar D}$ {\it up to a 
(multiplicative) integration constant}. More precisely, the condition
\eqref{indeed} implies that $d\log q_\Q\in \bigl(1 + z \Q[[z]]\bigr)\frac {dz}{z}$, 
which we can integrate to a local coordinate $q_\Q$ on $\bar D$ that
is well-defined up to overall normalization. We write $\nabla_{t_\Q}$ for the 
contraction with the corresponding logarithmic vector field $q_\Q\frac{d}{dq_\Q}$
which is independent of that normalization.

To obtain the other half of the basis, we first normalize the pairing $\langle 
\cdot,\cdot\rangle$ by choosing a constant ($\nabla$-parallel) section ${\bf 1}_6$ of 
$H^6_{\rm log}(\bar Y/\bar D)$. This trivialization allows us to identify $e^3$ 
as a section of the rank-one subbundle $\bar F^3$ such that 
$\langle e^0,e^3\rangle = {\bf 1}_6$. Finally, we put $e^2:=\nabla_{t_\Q} e^3$,
so that by compatibility of the pairing with the connection we obtain in
the familiar fashion
\begin{equation}
\label{familiar}
\langle e^1,e^2\rangle = \langle e^1,\nabla_{t_\Q}e^3\rangle
=-\langle\nabla_{t_\Q} e^1,e^3\rangle = \langle e^0,e^3\rangle={\bf 1}_6
\end{equation}
Thus, we learn that $(e^s)$ is a symplectic basis of $\bar\calh$ over $\bar D$.
One also checks as usual that $\nabla_{t_\Q} e^2$ is proportional to $e^1$, 
and defines the ``rational Yukawa coupling'' $\calc_\Q\in\calo_{\bar D}$ by
\begin{equation}
\label{dropsout}
\nabla_{t_\Q} e^2 = -e^1\otimes \calc_\Q 
\end{equation}
We emphasize that the seeming ease in obtaining the basis $(e^s)$ is a consequence 
of solving the differential equation of parallelism over the field of rational numbers. 
Complexification of the basis will yield a basis of the complex Hodge bundle
that agrees with the one used in subsection \ref{VHS}, {\it up to the normalization}
of $e^0$ and ${\bf 1}_6$. The normalization of $e^0$ (though not that of ${\bf 1}_6$) drops out
of \eqref{dropsout}, and both are fixed by the topological integrality that
underlies the complex VHS, and which is expressed through relations of the type 
\eqref{periodmatrix}. With this out of the way, the single remaining difficulty 
in identifying $\calc_\Q$, obtained from the local solutions of the differential 
equation, with the complex power series from section \ref{VHS} is the proper 
normalization of $q_\Q$. 
 Under our standing assumption that $q_\C'(0)\in\Q$, the results so far
imply that $q_\Q=q_\C$ up to a rational factor, thus proving the first statement in
Lemma \ref{rationalpowers}. 

\medskip

The extension by the algebraic cycle is readily included. As explained in subsection 
\ref{intdef}, we first pass to an $r$-fold of $\bar D$, $\bar{\hat D}:=\spec
\Q[[\hat z]]\to \bar D$ with $z=\hat z^r$, such that the composition $\bar\pi\circ
i:\bar C\to\bar{\hat D}$ is unramified and irreducible over $\Q$. An important novelty
is that we do not assume $\bar C$ to be defined over $\Q$. Namely, $\bar C$ could
become reducible after base change to the algebraic closure $\bar \Q$.  
This affords $\bar C^{\bar \Q}=\bar C\times_{\spec\Q}\spec\bar\Q$ 
with an action of the
absolute Galois group ${\rm Gal}(\bar \Q/\Q)$ and we can identify the field of 
definition, $K$, as the number field invariant under the subgroup of ${\rm Gal}(\bar \Q/\Q)$ 
that fixes the components of $\bar C^{\bar \Q}$.

We then define the ``extended and continued rational Hodge bundle'' via
\begin{equation}
\bar{\hat\calh} = \check{H}^3_{\rm log}\bigl((\bar Y\setminus\bar C)/{\bar{\hat D}}\bigr)
\end{equation}
This fits into an exact sequence
\begin{equation}
\label{ratext}
\bar\calh \to\bar{\hat\calh} \to \bar{\hat \cali}
\end{equation}
with $\bar\calh$ from \eqref{vialog} and
\begin{equation}
\bar{\hat \cali}= (\bar\pi\circ i)_*(\calo_{\bar C})
\end{equation}
$\bar{\hat \cali}$ is a vector bundle over $\bar{\hat D}$ of rank equal to the degree 
$d=[K:\Q]$. In particular, defining $\hat V := \bar{\hat\calh}_a$ we can write 
the extension of the fiber at $a$ as
\begin{equation}
\label{extfibata}
V\to\hat V\to K
\end{equation}
where we view $K$ either as a $\Q$-vector space of dimension $d$, or a $K$-vector space
of dimension $1$. The latter point of view is more convenient to study the differential
equations, so we adopt it in what follows. In other words, we now study the differential
equation over the ``disk with scalar extension'', $\bar{\hat D}^K=\bar{\hat D}
\times_{\spec\Q}\spec K=\spec K[[\hat z]]$.
We note that for reasons of degree, the extension of the filtration satisifies
\begin{equation}
\label{freedom}
\bar{\hat F}^2/{\bar F}^2 = \bar{\hat \cali}
\end{equation}
and that the residue $\hat N_{\rm dR}$ of the extended Gauss-Manin connection acts 
trivially on the quotient $K=\hat V/V$in \eqref{extfibata}. 

\begin{lemma} 
There exists a section $f$ of $\bar{\hat F}^2$ that satisfies 
\begin{equation}
\label{allowus}
\nabla_{t_\Q} f \in e^1\otimes\calo_{\bar{\hat D}^K}
\end{equation}
and whose restriction to the closed point $a\in\bar{\hat D}^K$ generates $K$.
\end{lemma}
\begin{proof}
Indeed, by Griffiths transversality and
parellelism of the image in $\bar{\hat\cali}$, we have $\nabla_{t_\Q} \bar{\hat F}^2
\subset \bar{F}^1$, which is spanned by $(e^s)$ wih $s>0$. Since $\nabla_{t_\Q}$ 
maps $\bar F^2$ surjectively onto $\bar\calh/\bar F^2$, we can use the freedom
\eqref{freedom} to fix $f$ such that \eqref{allowus} is satsified. (See
\eqref{specify} for the complex analogue of this construction.)
\end{proof}
In close analogy to the absolute case, $f$ is unique up to multiplication
by a non-zero constant, and, defining the ``$K$-rational Griffiths infinitesimal
invariant'' $\cald_K\in\calo_{\bar{\hat D}^K}$ by
\begin{equation}
\nabla_{t_\Q} f = -e^1\otimes \cald_{K},
\end{equation}
we can choose the normalization of $f$ such
that after complexification and choice of embedding $K\hookrightarrow\C$,
$\cald_K$ agrees with $\cald_k$ from \eqref{ontheotherhand}, \eqref{griff}.

At this point, the proof of Lemma \ref{rationalpowers} is complete.\qed

\subsection{Fontaine-Lafaille modules}

The de Rham cohomology over $\Q$ that we described in the previous subsection
can be endowed with further structure in several different ways. One possibility
is to complete our schemes with respect to the standard Archimedean norm, and 
after algebraic closure we obtain the standard topology of a family of complex
manifolds $\bar Y_\C\to\bar D_\C$, later extended by the algebraic
cycle ${\bar C}_\C\subset {\bar Y}_\C$. The cohomology (relative to $D_\C$)
of the constant sheaf of integers over this topology is well-behaved 
outside of the boundary point $a_\C=\bar D_\C\setminus D_\C$. The 
resulting local system enriches the de Rham cohomology 
${\calh_\C}\to D_\C$ (resp.\ ${\hat\calh_\C}\to{\hat D}_\C$) to the variation of
Hodge structure (localized in the neighborhood of $a_\C$) that we described in 
section \ref{VHS}.\footnote{except that there were no $\C$-subscripts 
in that section to 
reduce cluttering.} As we have alluded to before, as far as the calculation of
invariants through the solution of differential equations is concerned, the
main import of the topological integral structure is the proper normalization
of the sections $e^0$ and ${\bf 1}_6$ (and later $f$), and the multiplicative 
normalization of
the canonical coordinate $q_\Q$ (which is equivalent to an additive normalization 
of $e^1$). This normalization can be accomplished by interpreting the power series 
as representatives of extension classes in the category of mixed Hodge 
structures.

\medskip

As an alternative to the complex topology, we can, for each ``good'' prime $p\in S$,
complete our family with respect to the $p$-adic topology. This process is less
familiar to some, but we are unable to overburden this paper with a full 
recollection. 
The ultimate idea is that the algebraic de Rham cohomology over
$\Q_p$ can be provided with
an additional ``integral'' substructure by identifying it as the cohomology
of a constant sheaf with respect to a somewhat subtle ``crystalline'' topology
on the geometry in question, tensored with a suitable ``period ring''. 
The main feature at the end of this process is the identification of the
cohomology as a module over the absolute
Galois group ${\rm Gal}(\bar\Q_p/\Q_p)$, which ends up taking the role
played by the singular cohomology in the more familiar complex case.
(See also \cite{sv2} and \cite{shapiro} for some additional information.)

We actually do not need the full machinery of this theory, ref.\ \cite{faltings},
but only the action of the Frobenius element on the de Rham cohomology,
which can be defined already from the scheme over $\Z_p$. In the situation
relevant to us, the notion that is the $p$-adic analogue of the complex 
variation of Hodge structure is identified \cite{vadim} as a Fontaine-Lafaille
module over a $p$-adic scheme. For the continuation 
of our semi-stable morphism $\bar Y\to\bar D$ over $\Q$ to a $p$-adic family 
$\bar Y_p\to\bar D_p$, this amounts to the following.

Among the early steps, one needs to equip the $p$-adic disk $\bar D_p=\spec\Q_p[[z]]$
with a continuous lift of the Frobenius endomorphism $z\mapsto z^p$ at 
the closed point $a_p\in \bar D_p$. (Recall that by assumption, $z$
is integral at $p\in S$.) Very concretely, we have $\frob_p(x)=x$
for $x\in\Z_p$ and $\frob_p(z)=z^p(1 + p\, \eta(z))$ for some,
not necessarily zero, $\eta(z)\in
\Z_p[[z]]$. It might seem that this step involves some abitrary
choices beyond those present in the rational or complex algebraic
setting. This is however not the case, as emphasized in \cite{vadim},
different lifts being related by canonical isomorphisms.
In fact, the endpoint of the $p$-adic construction is precisely the 
identification of a canonical coordinate $q_p$ in which Frobenius
acts by $\frob_p(q_p)=(q_p)^p$, as assumed in the statements of the Theorems 
\ref{previous} and \ref{origin}. To account for the existence of this 
integral structure on the disk, it is convenient to substitute 
$\spec\Z_p[[z]]$ (with distinguished $a_p=(z)$) for $\bar D_p$
in the following.

Let us now denote by $\bar \calh_p\to\bar D_p$ the vector bundle
(as before, of rank $4$) of logarithmic de Rham cohomology over $\Z_p$.
As over $\Q$, it continues to possess a filtration $\bar F^*$,
flat connection $\nabla$ and pairing $\langle\cdot,\cdot\rangle$
with similar properties as before.
The essential new ingredient is a canonical lift of
the Frobenius morphism,
\begin{equation}
\Phi_p:(\frob_p)^* \bar\calh_p \to \bar\calh_p
\end{equation}
that is parallel in the sense that
\begin{equation}
\label{parallel}
\nabla \circ \Phi_p = \Phi_p\circ \nabla\,,
\end{equation}
compatible with the filtration in the sense that
\begin{equation}
\label{divisible}
\Phi_p(\frob_p^* \bar F^s) \subset p^s \bar\calh_p
\quad\text{ and }
\sum_{s} p^{-s} \Phi_p(\frob_p^*\bar F^s)=\bar\calh\,,
\end{equation}
and with the pairing in the sense that
\begin{equation}
\langle \Phi_p\circ\frob_p^* (u),\Phi_p\circ\frob_p^*(v) \rangle 
= p^3 \frob_p^* \langle u,v\rangle
\end{equation}
This last equation being a transcription of the statement that
\begin{equation}
H_{\rm log}^6(\bar Y_p/\bar D_p) \cong \Z_p(-3)
\end{equation}
is an instance of a Fontaine-Lafaille module of the type
\begin{equation}
\Z_p(-k) = \bigl(\bar F^k=\calo_{\bar D_p}, \bar F^{k+1}=0,
\Phi_p=p^k\cdot{\rm id}\bigr)\,,
\end{equation}
which is the $p$-adic version of the Hodge-Tate structure $\Z(-k)$.

A central observation of \cite{vadim} in this context is that, in the category
of Fontaine-Lafaille modules over the {\it punctured} disk $D_p=\spec
\bigl(\Z_p((t))\bigr)$,
\begin{equation}
{\rm Ext}^1_{{\rm MF}(D_p)}(\Z_p(-k),\Z_p(-k+1)) \cong {\hat\calo^*}(D_p)
\end{equation}
where ${\hat\calo^*}(D_p)$ is the $p$-adic completion of ${\calo^*}(D_p)$,
the invertible functions on $D_p$. This statement is the analogue of \eqref{same}
in the complex case and implies that {\it the power series parameterizing
extensions of Fontaine-Lafaille modules have integral coefficients}, provided of 
course that {\it they are calculated with respect to an integral basis and 
coordinate}.

As a final ingredient, we require the residue $N_{\rm dR}$ of the flat connection 
at $a_p$ in order to induce a weight filtration on the limiting Fontaine-Lafaille 
module $V_p=V\otimes\Q_p$ at $a_p$. Notice that the F-L structure on $V_p$ (especially 
the Frobenius) depends on the choice of coordinate (namely, through the 
choice of Frobenius, $\frob_p$) on $\bar D_p$.

Given all this, it is shown in \cite{vadim} that the Fontaine-Lafaille module
\begin{equation}
\call_p=(\bar\calh_p,\Phi_p,\bar F^*,\langle\cdot,\cdot\rangle)
\end{equation}
coming from our Calabi-Yau threefold family has a composition series 
very much analogous to that discussed in section \ref{VHS} in the 
complex case. We shall not retrace these steps here. The essential result
is the identification of the $p$-adic canonical coordinate, $q_p$, and
the $p$-adic Yukawa coupling, $\calc_p$ as representatives of extensions 
classes in the category ${\rm MF}(D_p)$ of Fontaine-Lafaille modules over 
$D_p$, with respect to a distinguished basis of sections $(e^s_p\in\bar F^s)$ 
of the Hodge filtration. This data satisfies the same differential equations
as over $\Q$ (and $\C$).

\medskip

To include the algebraic cycle, we turn to working over the extended 
disk at $p$, $\bar{\hat D}_p^K=\spec\G_p[[\hat z]]$. Notice that in this
case, Frobenius (still denoted $\frob_p$) acts non-trivially already on 
the residues at $a_p$, cmp.\ eq.\ \eqref{basfrob}. With the extended Hodge bundle 
$\bar{\hat\calh}_p\to\bar{\hat D}_p^K$, the $p$-adic continuation of
eq.\ \eqref{ratext} (whose complex version is eq.\ \eqref{hatext})
takes the form
\begin{equation}
\bar{\hat\cali}_p := \bar{\hat\calh}_p/\bar\calh_p \cong \G_p(-2) 
\end{equation}
where 
\begin{equation}
\G_p(-2)=\bigl(\bar F^2 = \calo_{\bar{\hat D}_p^K},\bar F^3=0,\Phi_p=p^2\cdot\frob_p\bigr)
\end{equation}
is the Fontaine-Lafaille module of rank $1$ over $\bar{\hat D}_p^K$
with Frobenius inherited from $K_p$.
On the preimage of $\bar{\hat\cali}_p$ in $\bar{\hat\calh}_p$, this lifts to
\begin{equation}
\Phi_p = p^2\cdot \frob_p  \bmod\bar F^2
\end{equation}
so that similarly to
$\Q$ or $\C$, we can obtain a section $f_p$ whose restriction to $a_p$ generates 
$\hat V_p/V_p\cong K_p$ and such that
\begin{equation}
\eqlabel{by}
\Phi_p(f_p) - p^2 f_p \in \spann_{\calo_{\bar{\hat D}_p^K}} (e^0_p,e^1_p)
\end{equation}
This allows us to identify the $p$-adic infinitesimal invariant $\cald_p$,
\begin{equation}
\label{pder}
\nabla_{t_p} f_p = -\cald_p e^1_p
\end{equation}
as the derivative of the extension class 
\begin{equation}
\nu_p\in {\rm Ext}^1_{{\rm MF}({\hat D}^K_p)}(\G_p(-2),\call_p)
\end{equation}
in the category of Fontaine-Lafaille modules over the punctured extended $p$-adic
disk (cmp.\ \eqref{griff}). Specifically,
\begin{equation}
\cald_p = \langle\nabla_{t_p}^2\nu_p,e^3_p\rangle =
-\langle \nabla_{t_p}\nu_p,e^2_p\rangle
\end{equation}

\subsection{Identification of extension classes}
\label{motivic}

We have now defined the $p$-adic power series $q_p$, $\calc_p$,
and $\cald_p$ parameterizing the composition of the (extended) Fontaine-Lafaille 
module associated to our Calabi-Yau scheme (with cycle) over the disk. For clarity, 
we let $q_\C$, $\calc_\C$ and $\cald_\C$ be the complex power series that were 
introduced in section \ref{geosit} without the subscript. (We'll also add that
subscript to the cohomology basis to write $(e^s_\C,f_\C)$.)
\begin{proposition}
\label{equality} In $\Q[[z]]$, $\Q[[q]]$, and $K[[\hat q]]$ respectively,
\begin{equation}
\begin{split}
q_p &= q_\C  \\
\calc_p &= \calc_\C \\
\cald_p &= \cald_\C
\end{split}
\end{equation}
\end{proposition}
\begin{proof}
We notice that these power series satisfy the same differential equation over
$\C$ and $\Q_p$ (or $K_p$ in the extended case of $\cald$) as they
do over $\Q$ (or $K$). In particular, these differential equations imply
that $d\log q_p=d\log q_\C$ and that the other two equations hold up to an overall
factor (in $\Q^*$ or $K^*$, respectively). To show equality, we need to
compare the normalization of the cohomology bases $(e^s_\C,f_\C)$ of the VHS
and $(e^s_p,f_p)$ of the F-L structure. Again by virtue of the differential 
equations, and duality
with respect to the pairing, it is in fact sufficient to establish equality 
for the subvariations spanned by $(e^0,e^1)$ in the two cases.
We refer to section 4 of \cite{vadim} for the proof of this statement.
\end{proof}

\section{Integrality Proofs}
\label{proof}

By using the results reviewed in the previous section, Theorems \ref{previous}
and \ref{origin} follow from a couple lines of simple algebra. Thanks to
Proposition \ref{equality}, it is enough to verify the $p$-adic integrality of
the $p$-adically defined functions. We shall drop the subscript
$p$ from most of the notation in what follows.

\medskip

\subparagraph{We recapitulate some notation:}

$K$ is an algebraic extension of $\Q$, $\G$ the ring of integers in $K$.
$p$ is a rational prime, and $\G_p$ the $p$-adic completion of $\G$.
$\frob_p$ is Frobenius as defined in subsection \ref{definitions}.

$\bar D$ is a formal disk over $\Z_p$, $\bar{\hat D}^K$ its $r$-fold
cover, extended over $\G_p$. Given a local coordinate $z$
(\ie, an identification $\bar{D}\cong\spec\Z_p[[z]]$), we obtain an
endomoprhism of $\bar D$ lifting Frobenius by putting
\begin{equation}
\frob_p^{(z)} (z) = {z}^p
\end{equation}
We have added the superscript to emphasize the dependence on the local 
coordinate. One is tempted to drop it when $z$ is clear from the context.
But we will do so only after replacing $z$ with the canonical coordinate
$q$. In the lift to $\bar{\hat D}^K$, $\frob_p^{(z)}$ acts non-trivially
also on the coefficients in the residue ``product-of-fields'' $K_p$.

$\bar{\hat\calh}$ is a bundle of $\G_p$-modules over $\bar{\hat D}^K$,
with a filtration $\bar{\hat F}^*$ by bundles of $\G_p$ submodules.

$\Phi_p^{(z)}$ is a bundle map $(\frob_p^{(z)})^*\bar{\hat\calh}\to
\bar{\hat\calh}$ lifting $\frob_p^{(z)}$ to the extended Hodge bundle.
We usually identify $\Phi_p^{(z)}$ with $\Phi_p^{(z)}\circ (\frob_p^{(z)})^*$.

$\Phi_p^{(z)}$ preserves the weight filtration and is divisible by $p^s$ 
on $\bar{\hat F}^s$. $\Phi_p^{(z)}$ is also compatible with the pairing
$\langle\cdot,\cdot\rangle: \bar\calh\times \bar\calh\to \Z_p(-3)$
(namely, with $\Phi_p^{(z)} = p^3 {\rm id}$ on $\Z_p(-3)$).

\subsection{Integrality of cohomology basis}

The first item on the list is to verify that the basis element $e^0$,
which is defined in \ref{overQ} as the parallel section of the rational
bundle $\bar\calh$ with $(e^0)_a\in W_0=\ker(N_{\rm dR})$, is $p$-adically
integral, \ie, continues to a section $e^0_p\in \bar F^0\bar\calh_p$.
To see this (cf.\ Lemma 7 of \cite{vadim}), one first notices that $(e^0)_a$ 
is eigenvector 
of $(\Phi_p^{(z)})_a$ with eigenvalue of square $1$ (this follows from the
invariance of the pairing and $N_{\rm dR} \circ (\Phi^{(z)})_a = p \cdot
(\Phi^{(z)})_a \circ
N_{\rm dR}$). Then, letting $\tilde e^0$ be any section of $\bar F^0\bar\calh_p$
with $(\tilde e^0)_a=(e^0)_a$, one observes that 
\begin{equation}
\lim_{k\to\infty}\bigl(\Phi^{(z)}\bigr)^{2k} (\tilde e^0)
\end{equation}
is a parallel and integral section that agrees with $(e^0)_a$ at $a$. By
uniqueness of the solution of the differential equation, this is $e^0$.

The integrality of the remainder of the cohomology basis follows from
its construction via duality and taking derivative with respect to the 
canonical coordinate, whose integrality is established next.

\subsection{Integrality of mirror map}

$q_p$ is the class of the extension
\begin{equation}
\Z_p(0) \to \call_2 \to \Z_p(-1)
\end{equation}
in the category of Fontaine-Lafaille modules over the formal disk 
$\spec\Z_p[[z]]$ over $\Z_p$. Lemma 5 of \cite{vadim} (a.k.a.\ the 
Dwork integrality Lemma) implies $q_p\in z\Z_p[[z]]$.

\subsection{Integrality of Griffiths-Yukawa coupling}

The rest of the Fontaine-Lafaille structure and the integrality of the
Yukawa coupling is best analyzed in the canonical coordinate $q$. From now
on, we drop the superscript from Frobenius. The following calculation first 
appeared in \cite{ksv}.

We let $(p^s m^s_t)$ be the matrix representing Frobenius with respect to
the basis $(e^s)$. Namely, we write
\begin{equation}
\eqlabel{matrixnot}
\Phi_p((\frob_p)^*(e^s)) = p^s \sum_{t=0}^s m^s_t e^t
\end{equation}
By the above, $m_t^s\in \Z_p[[z]]$.

In the same basis, the Gauss-Manin connection contracted with the canonical
vector field $t=q\partial_q$ has the representation
\begin{equation}
\nablat = \begin{pmatrix}
0 & 0 & 0 & 0 \\
-1 & 0 & 0 & 0 \\
0 & -\calc & 0 & 0 \\
0 & 0 & 1 & 0 
\end{pmatrix}
\end{equation}
We already know that $\calc\in\Z_p[[z]]$. Since $\calc$ is the 
logarithmic derivative of the extension class of $\call_4/\call_0$ in 
${\rm Ext}_{\rm MF}(\Z_p(-2),\Z_p(-1))$ (\cf, \eqref{recognize} for the corresponding 
complex statement), Dwork's lemma implies that $\frob_p(\calc)-\calc=
\delta\varphi$ for some $\varphi\in\Z_p[[q]]$. We wish to improve this
to the statement that
\begin{equation}
\eqlabel{improve}
\frob_p(\calc)-\calc=\delta^3\psi
\end{equation}
for $\psi\in\Z_p[[q]]$. To evaluate
\begin{equation}
\eqlabel{hhh}
\nablat \Phi_p = p \Phi_p\nablat
\end{equation}
we calculate
\begin{equation}
\eqlabel{first}
\begin{split}
\nablat\Phi_p & =
\delta
\begin{pmatrix}
 m^0_0 & 0 & 0 & 0\\
p m^1_0 & p m^1_1 & 0 & 0 \\
p^2 m^2_0 & p^2 m^2_1 & p^2 m^2_2 & 0 \\
p^3 m^3_0 & p^3 m^3_1 & p^3 m^3_2 & p^3 m^3_3 
\end{pmatrix}
+
\begin{pmatrix}
 m^0_0 & 0 & 0 & 0\\
p m^1_0 & p m^1_1 & 0 & 0 \\
p^2 m^2_0 & p^2 m^2_1 & p^2 m^2_2 & 0 \\
p^3 m^3_0 & p^3 m^3_1 & p^3 m^3_2 & p^3 m^3_3 
\end{pmatrix}
\cdot
\begin{pmatrix}
0 & 0 & 0 & 0 \\
-1 & 0 & 0 & 0 \\
0 & -\calc & 0 & 0 \\
0 & 0 & 1 & 0 
\end{pmatrix}
\\
&=
\delta
\begin{pmatrix}
 m^0_0 & 0 & 0 & 0\\
p m^1_0 & p m^1_1 & 0 & 0 \\
p^2 m^2_0 & p^2 m^2_1 & p^2 m^2_2 & 0 \\
p^3 m^3_0 & p^3 m^3_1 & p^3 m^3_2 & p^3 m^3_3 
\end{pmatrix}
+
\begin{pmatrix}
0 & 0 & 0 & 0 \\
-p m^1_1 & 0 & 0 & 0 \\
-p^2 m^2_1 & -p^2 \calc m^2_2 & 0 & 0 \\
-p^2 m^3_1 & -p^3 \calc m^3_2 & p^3 m^3_3 & 0 
\end{pmatrix}
\end{split}
\end{equation}
and
\begin{equation}
\eqlabel{second}
\begin{split}
p \Phi_p \nablat &= p 
\begin{pmatrix}
0 & 0 & 0 & 0 \\
-1 & 0 & 0 & 0 \\
0 & -\frob_p(\calc) & 0 & 0 \\
0 & 0 & 1 & 0 
\end{pmatrix}
\cdot
\begin{pmatrix}
 m^0_0 & 0 & 0 & 0\\
p m^1_0 & p m^1_1 & 0 & 0 \\
p^2 m^2_0 & p^2 m^2_1 & p^2 m^2_2 & 0 \\
p^3 m^3_0 & p^3 m^3_1 & p^3 m^3_2 & p^3 m^3_3 
\end{pmatrix}
\\
&=
\begin{pmatrix}
 0 & 0 & 0 & 0 \\
-p m^0_0 & 0 & 0 & 0 \\
-p^2 \frob_p(\calc) m^1_0 & -p^2 \frob_p(\calc) m^1_1 & 0 & 0 \\
p^3 m^2_0 & p^3 m^2_1 & p^3 m^2_2 & 0 
\end{pmatrix}
\end{split}
\end{equation}
The compatibility of $\Phi_p$ with the pairing in cohomology takes the form
\begin{equation}
\eqlabel{compatibility}
m^t_s I^{ss'} m_{s'}^{t'} = p^3 I^{tt'}
\end{equation}
where 
\begin{equation}
I = 
\begin{pmatrix}
0 & 0 & 0 & 1 \\
0 & 0 & 1 & 0 \\
0 & -1 & 0 & 0 \\
-1 & 0 & 0 & 0
\end{pmatrix}
\end{equation}
Now, by equating the diagonal terms in \eqref{first} and \eqref{second}, we find that
$\delta m^s_s=0$ for $s=0,1,2,3$. At the top of the first lower diagonal, we learn that
\begin{equation}
\eqlabel{infactvanishes}
\delta m^1_0 = m^1_1-m^0_0
\end{equation}
Since $m^1_0\in\Z_p[[q]]$ (in particular, it contains no negative powers of
$q$), evaluation at $q=0$ shows that the constant $m^1_1-m^0_0$ in fact vanishes.
Continuing down, we find that $m^0_0=m^1_1=m^2_2=m^3_3$. 

Putting $t=0$, $t'=3$ in \eqref{compatibility}, we find that $1=m^0_0m^3_3= (m^0_0)^2$. 
Let us assume that $m^0_0=1$. (The case $m^0_0=-1$ can be treated {\it mutatis mutandis}.) 
Then the remaining entries of \eqref{compatibility} become
$m^3_2+ m^1_0 =0$ and $m^2_0 + m^2_1 m^3_2- m^3_1=0$.

Returning to \eqref{infactvanishes}, we see that $m^1_0$ is a constant. In fact,
by the results of section \ref{motivic}, this constant is $0$.

Finally, the lower left $2\times 2$ square of \eqref{first} and \eqref{second} becomes
\begin{equation}
\begin{split}
\delta m^2_1 &= \calc-\frob_p(\calc)\\
\delta m^2_0 &= m^2_1 \\
\delta m^3_1 &= m^2_1 \\
\delta m^3_0 &= m^3_1 + m^2_0 
\end{split}
\end{equation}
Put together, this gives
\begin{equation}
\frob_p(\calc) - \calc = -\frac 12 \delta^3 m^3_0
\end{equation}
The claim follows since $p\neq 2$. This ends the proof of Theorem \ref{previous}.
\qed

\subsection{Integrality of infinitesimal invariant}

This calculation takes place over the extended disk in canonical coordinates, localized
at $p$ in the sense of \eqref{unconv}, $\bar{\hat D}^K_p := \spec \calo_p[[\hat q]]$.

By eq.\ \eqref{by}, we can write
\begin{equation}
\Phi_p(f) = p^2 f + p^2 n_{0} e^0 + p^2 n_{1} e^1
\end{equation}
with $n_{0}, n_{1}\in \calo_p[[\hat q]]$. On the other hand (cf.\ \eqref{ontheotherhand},
\eqref{pder}),
we have 
\begin{equation}
\nablat f = - \cald e^1
\end{equation}
The equality \eqref{hhh} becomes
\begin{equation}
\begin{split}
\nablat \Phi_p f - p \Phi_p\nablat f &= p^2 (\delta n_{0}-n_{1})e^0 + 
p^2 (\delta n_{1} -\cald + \frob_p(\cald)) e^1 =0
\end{split}
\end{equation}
(where we used $m^1_1=1$ and $m^1_0=0$ from the previous subsection).
As a result,
\begin{equation}
\begin{split}
\delta n_{0} &= n_{1} \\
\delta n_{1} &= \cald-\frob_p(\cald)
\end{split}
\end{equation}
and by combining the two, we find
\begin{equation}
\frob_p(\cald) - \cald = -\delta^2 n_{0}
\end{equation}
This concludes the proof of Theorem \ref{origin}.
\qed

\begin{acknowledgments}
The research of V.V.\ was supported in part by Laboratory of Mirror Symmetry NRUHSE,
RF government grant, ag.\ number 14.641.31.0001.
The research of J.W.\ was supported in part by an FRQNT New Researcher grant, 
an NSERC discovery grant, a Tier II 
Canada Research Chair at McGill University, and by the German Research Foundation
in the framework of the Excellence Initiative through MATCH at Heidelberg University.
The research of A.Sch.\ was supported in part by an NSF grant. 
\end{acknowledgments}






\begin{thebibliography}{99}
\addcontentsline{toc}{section}{References}
\renewcommand{\itemsep}{-.2cm}
\colorlinksblue
\small


\bibitem{ksv} 
  M.~Kontsevich, A.~S.~Schwarz and V.~Vologodsky,
  ``Integrality of instanton numbers and p-adic B-model,''
  Phys.\ Lett.\ B {\bf 637}, 97 (2006)
  [\hepth{0603106}].

\bibitem{sv1}
  A.~S.~Schwarz and V.~Vologodsky,
  ``Frobenius transformation, mirror map and instanton numbers,''
  Phys.\ Lett.\ B {\bf 660}, 422 (2008)
  [\hepth{0606151}].

\bibitem{vadim}
V.~Vologodsky,
``On the $N$-integrality of instanton numbers,''
preprint

\bibitem{sv2} 
  A.~Schwarz and V.~Vologodsky,
  ``Integrality theorems in the theory of topological strings,''
  Nucl.\ Phys.\ B {\bf 821}, 506 (2009)
  [\arxiv{0807.1714}{hep-th}].

\bibitem{arithmetic}
J.~Walcher, ``On the arithmetic of D-branes superpotentials,''
Comm.\ Num.\ Th.\ Phys.\ {\bf 6}, no. 2, 279--337 (2012) [\arxiv{1201.6427}{hep-th}]

\bibitem{sherridan}
N.~Sherridan
``Homological mirror symmetry for a Calabi-Yau hypersurface in projective space,''
Invent.\ Math.\ {\bf 199}, 1--186 (2015)

\bibitem{akv}
 M.~Aganagic, A.~Klemm and C.~Vafa,
  ``Disk instantons, mirror symmetry and the duality web,''
  Z.\ Naturforsch.\  A {\bf 57}, 1 (2002)
  [\hepth{0105045}].

\bibitem{svw1}
A.~Schwarz, V.Vologodsky, and J.~Walcher, 
``Framing the Di-logarithm (over Z),''
Contribution to Proceedings of String-Math 2012, Bonn
[\arxiv{1306.4298}{hep-th}].

\bibitem{agva}
M.~Aganagic and C.~Vafa
``Mirror symmetry, D-branes and counting holomorphic discs,''
 \hepth{0012041}.

\bibitem{opening}
  J.~Walcher,
  ``Opening mirror symmetry on the quintic,''
  Commun.\ Math.\ Phys.\  {\bf 276}, 671 (2007)
  [\hepth{0605162}].

\bibitem{normal}
  D.~R.~Morrison and J.~Walcher,
  ``D-branes and Normal Functions,''
  Adv.\ Theor.\ Math.\ Phys.\  {\bf 13} (2009)
  [\arxiv{0709.4028}{hep-th}].

\bibitem{ferrari}
 F.~Ferrari,
  ``Galois symmetries in Super Yang-Mills Theories,''
  JHEP {\bf 0903}, 128 (2009)
  [\arxiv{0901.4079}{hep-th}].

\bibitem{cachazo}
 S.~K.~Ashok, F.~Cachazo and E.~Dell'Aquila,
  ``Children's drawings from Seiberg-Witten curves,''
  Commun.\ Num.\ Theor.\ Phys.\  {\bf 1}, 237 (2007)
  [\hepth{0611082}].

\bibitem{stienstra}
J.\ Stienstra, ``Mahler measure variations, Eisenstein series and instanton expansions,''
Mirror Symmetry V, 139--150,
AMS/IP Stud.\ Adv.\ Math., 38, Amer.\ Math.\ Soc., Providence, RI, 2006. 

\bibitem{granville}
A.~Granville, ``Arithmetic properties of binomial coefficients I:
binomial coefficients modulo prime powers,''
Canad.\ Math.\ Soc.\ Conf.\ Proc.\ {\bf 20} (1997), 253--275

\bibitem{jacobsthal}
V.~Brun, J.O.~Stubban, J.E.~Fjedlstad, R.~Tambs Lyche, K.E.~Aubert, W.~Ljunggren,
E.~Jacobsthal, 
``On the divisibility of the difference between two binomial coefficients,''
Den 11te Skandinaviske Matematikerkongress, Trondheim, 1949, pp.\ 42--54.
(1952)

\bibitem{gasu}
S.~Garoufalidis, P.~Kucharski and P.~Su\l kowski,
``Knots, BPS states, and algebraic curves,''
\arxiv{1504.06327}{hep-th}

\bibitem{gasu2}
P.Kucharski and P.~Su\l kowski,
``BPS counting for knots and combinatorics on words,''
\arxiv{1608.06600}{hep-th}


\bibitem{gusu}
  S.~Gukov and P.~Su\l kowski, 
``A-polynomial, B-model, and Quantization,''
  JHEP {\bf 1202} (2012) 070
  doi:10.1007/JHEP02(2012)070
  [\arxiv{1108.0002}{hep-th}].

\bibitem{guide}
D.~R. Morrison, 
``Mirror symmetry and rational curves on quintic threefolds:
a guide for mathematicians,'' 
J.\ Amer.\ Math.\ Soc.\ {\bf 6} (1993) 223--247,
  [\alggeom{9202004}]. 

\bibitem{delignelimit}
P.~Deligne, 
``Local behavior of Hodge structures at infinity,'' 
Mirror symmetry, II, AMS/IP Stud.\ Adv.\ Math., vol.~1, 
Amer.\ Math.\ Soc., 1997, pp.~683--699.

\bibitem{ggk2}
M.~Green, P.~Griffiths and M.~Kerr,
``Neron models and limits of Abel-Jacobi mappings,''
\href{http://www.math.wustl.edu/~matkerr/GGK1.pdf}{Compositio Math.\ {\bf 146} 
(2010), 288--366}

\bibitem{jefferson}
R.~A.~Jefferson, J.~Walcher,
``Monodromy of inhomogeneous Picard-Fuchs equations,''
Comm.\ Num.\ Th.\ Phys.\ {\bf 8}, 1 (2014)
[\arxiv{1309.0490}{hep-th}]

\bibitem{faltings}
G.~Faltings,
``Crystalline cohomology and $p$-adic Galois-representations,''
Algebraic Analysis, Geometry and Number Theory (Baltimore, MD 1988),
J.\ Igusa Ed., Johns Hopkins Univ.\ Press (1989), pp.\ 25--80


\bibitem{shapiro}
 A.~S.~Schwarz and I.~Shapiro,
  ``Supergeometry and arithmetic geometry,''
  Nucl.\ Phys.\ B {\bf 756}, 207 (2006)
  doi:10.1016/j.nuclphysb.2006.08.024
  [\hepth{0605119}].

\end{thebibliography}
\end{document}